\documentclass[12pt] {article}
\usepackage{amsmath}
\usepackage{amssymb}
\usepackage{amsthm}
\usepackage{graphicx}
\usepackage{color}
\newtheorem{theorem}{Theorem}
\newtheorem{lemma}{Lemma}
\newtheorem{remark}{Remark}
\newtheorem{assumption}{Assumption}
\title{Asymptotic stability of small solitary waves for nonlinear Schr\"odinger equations with electromagnetic potential in $ \RR^3$}
\author{Eva Koo \thanks{Department of Mathematics, The University of British Columbia, Vancouver} \\ evahk@math.ubc.ca}
\date{}
\def\aa{\ll x \rr^{-\sigma}}
\def\invaa{\ll x \rr^{\sigma}}
\def\invsqaa{\ll x \rr^{2\sigma}}
\def\rr{\rangle}
\def\ll{\langle}
\def\ee{\epsilon}
\def\dd{\delta}
\def\EE{\mathcal{E}}
\def\RR{\mathbb{R}}
\def\CC{\mathbb{C}}
\def\pp{\partial}
\def\HH{\mathcal{H}}

\begin{document}
\maketitle       

\begin{abstract}
We consider the nonlinear magnetic Schr\"odinger equation 
for $ u: \mathbb{R}^3 \times \mathbb{R} \to \mathbb{C} $,
\[ iu_t = (i \nabla + A)^2 u + V u + g(u) , u(x,0) = u_0(x),\]
where $ A :\mathbb{R}^3 \to \mathbb{R}^3 $ is the magnetic potential, 
$ V : \mathbb{R}^3 \to \mathbb{R} $ is the electric potential, 
and $ g = \pm | u |^2 u $ is the nonlinear term. We show that under 
suitable assumptions on the electric and magnetic potentials, if the initial data is 
small enough in $ H^1 $, then the solution of the above equation 
decomposes uniquely into a standing wave part, which converges as $t \to \infty$ and a dispersive part, which scatters.
\end{abstract}

\section{Introduction}    
Consider the nonlinear Schr\"odinger equation with magnetic and electric potentials for $ \psi(x,t): \RR^3 \times \RR \rightarrow \CC$, 
\begin{equation} \label{E1}
\left \{
 \begin{array}{lr}
 i \pp_t \psi  = (- \Delta + 2 i A \cdotp \nabla + i(\nabla \cdotp A) + V)\psi + g(\psi) \\
 \psi(x,0) = \psi_0(x) \in H^1(\RR^3)
 \end{array}
 \right.
\end{equation} 
where
\begin{equation} g(\psi) = \pm | \psi |^2 \psi .\end{equation}
Here, $A(x) = (A_1(x), A_2(x), A_3(x)) :\RR^3 \rightarrow \RR^3 $ and $ V(x): \RR^3 \rightarrow \RR $.
Equation (\ref{E1}) can be equivalently written as
\begin{equation}
 i \pp_t \psi = (i\nabla + A)^2 \psi + V \psi + g(\psi) 
\end{equation}
by replacing $V$ with $V - |A|^2$.
Here, $A(x) = (A_1(x), A_2(x), A_3(x))$ is the magnetic potential (also known as the vector potential), and $ V(x) $ is the electric potential (also known as the scalar potential). In this paper, we consider potentials $A(x)$ and $V(x)$ which decay to 0 as $|x| \to \infty$.

Equation (\ref{E1}) describes a charged quantum particle subject to external electric and magnetic fields, and a self-interaction (nonlinearity).
Such nonlinear Schr\"odinger equations find numerous physical applications, for example, in Bose-Einstein condensates and nonlinear optics. 

Just as for linear Schr\"odinger equations 
\begin{equation}
 i \pp_t \psi  = (- \Delta + 2 i A \cdotp \nabla + i(\nabla \cdotp A) + V)\psi ,
\end{equation}
an important role is played by standing wave solutions (or bound states)
\begin{equation} \label{Esol} 
\psi(x,t) = e^{iEt} Q(x) 
\end{equation}
of (\ref{E1}). Existence of standing waves to equation (\ref{E1}) for certain electrical and magnetic potentials was first proved in \cite{EL}.

Here we consider small solutions of the form (\ref{Esol}) which bifurcate from zero along an eigenvalue of the linear Hamiltonian operator
\begin{equation} 
  H = - \Delta + 2 i A \cdotp \nabla + i (\nabla \cdotp A) + V . 
\end{equation}
Physical intuition suggests that the ground-state standing wave (the one corresponding to the lowest eigenvalue $E$) should remain stable when the self-interaction (nonlinearity) is turned on, and indeed should become  \emph{asymptotically stable} (that is, nearby solutions should relax to the ground state by radiating excess energy to infinity -- see below for a more precise statement). When only one bound state is present, this was first proved in \cite{SW} for scalar potentials ($A \equiv 0$) and well-localized perturbations of the ground state. Later works addressed the more complicated situation of multiple bound states (e.g. \cite{TY} ,\cite{SW2}). For merely energy-space (i.e. $H^1(\RR^3)$) perturbations of the ground state, asymptotic stability was proved in \cite{GNT}, again for scalar potential ($A \equiv 0$). The main goal of the present paper is to prove asymptotic stability of the ground state, in the energy space, and in the additional presence of the magnetic field.

\begin{remark}
Our argument should also go through for nonlinearities $g(\psi) = \pm |\psi|^{p-1} \psi$ for $1 + 4/3 \leq p < 5$. For concreteness, we will work with $g(\psi) = \pm |\psi|^2 \psi$.
\end{remark}

In order to ensure the operator $H$
is self adjoint, we make the following assumption,
\begin{assumption} \label{self-adjoint} (Self-adjointness assumption)
We assume that each component of $A$ is a real-valued function in $ L^q + L^{\infty} $ for some $ q > 3 $ that $ \nabla \cdotp A \in L^2 + L^{\infty} $, and that $V$ is a real-valued function in $ L^2 + L^{\infty} $.
\end{assumption}
Then by Theorem X.22 of \cite{RSvol2}, the operator $ H $ is essentially self-adjoint on $ C_0^{\infty}(\RR^3)$.

\begin{assumption} \label{a2} (Spectral assumption)
We assume that $ H $ supports only one eigenvalue $ e_0 < 0 $ which is nondegenerate. We also assume 0 is not a resonance of $H$ (see e.g. \cite{EGS} for the definition of resonance).
\end{assumption}

We need the following assumption to show the existence and exponential decay of the nonlinear bound states.
\begin{assumption} \label{a_decay} (Assumptions for existence and exponential decay of nonlinear bound states)
We assume 
\begin{equation}  
  \| A \|_{L^q+L^\infty(|x| > R)} + \| V_{-} \|_{L^2+L^\infty(|x| > R)} \rightarrow 0  \; \text{as} \; R \rightarrow \infty  
\end{equation}
for some $ q > 3$.
\end{assumption}

Under the above assumptions, we have the following lemma on the existence and decay of nonlinear bound states.
Let $ \phi_0 > 0 $ be the positive, $ L^2 $-normalized eigenfunction corresponding to the eigenvalue $ e_0 $ of $H$.

\begin{lemma} \label{l1}
(Existence and decay of nonlinear bound states) For each sufficiently small $ z \in \CC $, there is a corresponding eigenfunction $Q[z] \in H^2 $ solving the nonlinear eigenvalue problem 
\begin{equation} \label{eNEP}
 HQ + g(Q) = EQ 
\end{equation}
with the corresponding eigenvalue $E[z] = e_0 + o(z) $ and $ Q[z] = z \phi_0 + q(z)$ with
\begin{equation} 
  q(z) = o(z^2), \;\;\;\; DQ[z] = (1,i) \phi_0 + o(z) \;\; \text{and} \;\; D^2Q[z] = o(1) \;\;\; \text{in} \;\; H^2 
\end{equation}
where we denote
\begin{equation}
  DQ[z] = (D_1 Q[z],D_2 Q[z]) = (\frac{\partial}{\partial z_1}Q[z],\frac{\partial}{\partial z_2}Q[z]), \;\; \text{and} \;\; z = z_1 + i z_2 .
\end{equation}

Furthermore, $ Q $ has exponential decay in the sense that
\begin{equation}
  e^{\beta |x|} Q \in H^1 \cap L^{\infty} 
\end{equation}
for some $ \beta > 0$ (independent of $z$).
\end{lemma}

Next, we need assumptions on $A$ and $V$ which ensure our linear Schr\"odinger evolution obeys some dispersive estimates.
For $ f, g \in L^2(\RR^3, \CC) $, define the real inner product $ \ll f, g \rr $ by
\begin{equation} \ll f, g \rr = \Re ( \int_{\RR^3} \bar{f}g dx ) . \end{equation}
Denote $ \ll x \rr = (1 + |x|^2)^{\frac{1}{2}} $ and fix $ \sigma > 4$. 
Let $P_c$ be the projection onto the continuous spectral subspace of $H$. Following \cite{EGS}, we have:
\begin{assumption} \label{a1} (Strichartz estimates assumption)
We assume that for all $ x, \xi \in \RR^3$,
\begin{equation} |A(x)| + \ll x \rr |V(x)| \lesssim \ll x \rr^{-1-\ee} , \end{equation}
\begin{equation} \ll x \rr^{1+\ee'} A(x) \in \dot{W}^{\frac{1}{2},6}(\RR^3) , \end{equation}
and 
\begin{equation} A \in C^0(\RR^3) \end{equation}
for some $ \ee > 0 $ and all sufficiently small $\ee' \in (0,\ee)$.
\end{assumption}

Define the space-time norm
\begin{eqnarray}
  \| \psi \|_X  
  & = & \| \aa \psi \|_{L_t^2 H_x^1} + \| \psi \|_{L_t^{3} W_x^{1,\frac{18}{5}}} + \| \psi \|_{L_t^{\infty} H_x^1}.
\end{eqnarray}
We can now state the main result, which says that all $H^1$-small solutions converge to a solitary wave (nonlinear bound state) as $t \to \infty$:

\begin{theorem} \label{t1} (Asymptotic stability of small solitary waves)
Let assumptions \ref{self-adjoint}, \ref{a2}, \ref{a_decay} and \ref{a1} hold.
For $0 \leq t < \infty$, every solution $ \psi $ of equation \eqref{E1} with initial data $ \psi_0 $ sufficiently small in $ H^1 $ can be uniquely decomposed as
\begin{equation} \psi(t) = Q[z(t)] + \eta(t) , \end{equation}
with differentiable $ z(t) \in \CC $ and $ \eta(t)  \in  H^1 $ satisfying $ \ll i \eta, D_1 Q[z] \rr = 0$, $ \ll i \eta, D_2 Q[z] \rr = 0 $
and
\begin{equation} \| \eta \|_{X} \lesssim \| \psi_0 \|_{H^1} , \;\;\;\; \| \dot{z} + i E[z] z \|_{L^1_t} \lesssim \| \psi_0 \|_{H^1}^2 . \end{equation}

Furthermore, as $ t \rightarrow \infty $,
\begin{equation} z(t)\exp(i \int_0^t E[z(s)] ds) \rightarrow z_+, \;\; E(z(t)) \rightarrow E(z_+) \end{equation}
for some $ z_+ \in \CC $ and
\begin{equation} \| \eta(t) - e^{-itH} \eta_+ \|_{H_x^1} \rightarrow 0 \end{equation}
for some $ \eta_+ \in H_x^1 \cap \text{Range}(P_c)$.
\end{theorem}

For comparison, consider the nonlinear Schr\"odinger equation with just a scalar potential $ V $,
\begin{equation} \label{E2}
 i \pp_t \psi = (- \Delta + V)\psi + g(\psi)
\end{equation}
for the same nonlinearity $ g $ as above, which is a special case of equation (\ref{E1}) with $ A = 0 $. The corresponding asymptotic stability result for (\ref{E2}) was obtained in dimension three in \cite{GNT}, in dimension one in \cite{TM1} and in dimension two in \cite{TM2, KZ}. Our approach for equation (\ref{E1}) will be similar to that in \cite{GNT}, which uses the Strichartz 
estimates
\begin{equation} \| e^{it(\Delta - V)} P_c \phi \|_{\tilde{X}} \lesssim \| \phi \|_{H^1} \end{equation}
and 
\begin{equation} \| \int_{-\infty}^t e^{i(t-s)(\Delta - V)} P_c F(s) ds \|_{\tilde{X}} \lesssim \| F \|_{L_t^2 W^{1,\frac{6}{5}}} \end{equation}
where $ \tilde{X} = L_t^{\infty} H^1 \cap L_t^2 W^{1,6} \cap L_t^2 L^{6,2} $, which are known to hold for a class of scalar potentials $ V $. Our approach will use the Strichartz estimates for $ H $ from \cite{EGS}. However, the proof of \cite{EGS} of the inhomogeneous Strichartz estimates
\begin{equation} \| \int_{-\infty}^t e^{i(t-s)H} P_c F(s) ds \|_{L_t^qL_x^p} \lesssim \| F \|_{L_t^{\tilde{q}'} L_x^{\tilde{p}'}} \end{equation}
for $ H = - \Delta + 2 i A \cdotp \nabla + i(\nabla \cdotp A) + V $ uses a lemma from \cite{MK} which does not hold for the endpoint case $ (q,p) = (2,6)$ or $(\tilde{q}, \tilde{p}) = (2,6) $. To overcome the lack of endpoint Strichartz estimates, we will use estimates in weighted spaces, as in \cite{TM1} and \cite{TM2}.

Section \ref{Linear estimates} is devoted to the proofs of the various linear dispersive estimates needed for the asymptotic stability argument. In addition to the estimates taken from \cite{EGS}, we need to establish estimates in weighted Sobolev spaces, which require some work. We will prove the following theorem.
\begin{theorem} \label{t7-2}
We say that $ (p,q)$ is Strichartz admissible if 
\begin{equation} \frac{2}{q} + \frac{3}{p} = \frac{3}{2} \;\;\; \text{with} \;\;\; 2 \leq p < 6 . \end{equation}
If $(q,p)$ and $(\tilde{p},\tilde{q}) $ are Strichartz admissible, then
\begin{eqnarray} 
  &       & \| \int_{0}^{t} e^{i(t-s)H} P_c F(s) ds \|_{L_t^q W_x^{1,p}} + \| \aa \int_{0}^{t} e^{i(t-s)H} P_c F(s) ds \|_{L_t^2 H_x^1} \\
  & \lesssim & \min ( \| \ll x \rr^{\sigma} F \|_{L_t^2 H_x^1} , \| F \|_{L_t^{\tilde{q}'} W_x^{1,\tilde{p}'}} ) . \end{eqnarray}
\end{theorem}

The asymptotic stability theorem is proved in section \ref{secMain}. Finally, the existence and decay of nonlinear bound states (Lemma \ref{l1}) is given in an appendix.

\section{Linear estimates} \label{Linear estimates}
The following lemmas \ref{l3} and \ref{l4} are from \cite{EGS}:

\begin{lemma} \label{l3} (Non-endpoint Strichartz estimates)
Under assumptions \ref{a1} and \ref{a2}, if $ (p,q)$ and $(\tilde{p},\tilde{q}) $ are Strichartz admissible, we have
\begin{equation} \| e^{itH} P_c f \|_{L_t^q L_x^p} \lesssim \| f \|_{L^2(\RR^3)} \end{equation}
and 
\begin{equation} \| \int_{-\infty}^{t} e^{i(t-s)H} P_c F(x) ds \|_{L_t^q L_x^p} \lesssim \| F \|_{L_t^{\tilde{q}'} L_x^{\tilde{p}'} } . \end{equation} 
\end{lemma}

Notice that the above does not include the $L_t^2$-norm. Fix $\sigma > 4$.

\begin{lemma} \label{l4} (Weighted homogeneous $L_t^2$ estimates)
Under assumptions \ref{a1} and \ref{a2}, we have
\begin{equation} \| \ll x \rr^{-\sigma} e^{-itH} f \|_{L^2_tL^2_x} \lesssim \| f \|_{L^2_x} , \end{equation}
and
\begin{equation} 
\sup_{\lambda \geq 0} \ll \lambda \rr \| \ll x \rr^{-\sigma} (H - (\lambda^2 + i 0))^{-1} \ll x \rr^{-\sigma} \|_{L^2 \rightarrow L^2} \lesssim 1 .
\end{equation}
\end{lemma}

The weighted resolvent estimate of lemma \ref{l4} implies weighted inhomogeneous estimates for the linear evolution:

\begin{lemma} \label{l6} (Weighted $ L_t^2 $ inhomogeneous estimates)
Under the assumptions of lemma \ref{l4},
\begin{equation} \| \aa \int_{0}^{t} e^{i(t-s)H} P_c \aa F(s) ds \|_{L_t^2L_x^2} \lesssim \| F \|_{L_t^2L_x^2}. \end{equation}
\end{lemma}

\begin{proof}
For simplicity we may restrict to times $t \geq 0$. By Plancherel, we have
\begin{eqnarray}
 &   & \| \chi_{\{ t \geq 0 \}} \aa \int_{0}^{t} e^{i(t-s)(H-i \ee)} P_c \aa F(s) ds \|_{L_t^2} \\
 & = & \| \int_{0}^{\infty} e^{it \tau} \aa (\int_{0}^{t} e^{i(t-s)(H-i \ee)} P_c \aa F(s) ds) dt \|_{L_\tau^2} 
\end{eqnarray}
Next, change the order of the $ds$ and $dt$ integral and use that 
\begin{eqnarray}
 &   & \int_{s}^{\infty} dt e^{it(H-\tau-i\ee)} P_c \aa F(s) \\
 & = & \frac{1}{i} (H-\tau-i\ee)^{-1}  e^{it(H-\tau- i\ee)}|_{t=s}^{t= \infty} P_c \aa F(s) \\
 & = & \frac{-1}{i} (H-\tau-i\ee)^{-1}  e^{is(H-\tau- i\ee)} P_c \aa F(s),
\end{eqnarray}
we get
\begin{eqnarray}
 &   & \| \chi_{\{ t \geq 0 \}} \aa \int_{0}^{t} e^{i(t-s)(H-i \ee)} P_c \aa F(s) ds \|_{L_t^2} \\
 & = & \| \aa \int_{0}^{\infty} ds e^{-is(H-i\ee)} \frac{-1}{i} (H-\tau-i\ee)^{-1}  e^{is(H-\tau- i\ee)} P_c \aa F(s) \|_{L_\tau^2} \\
 & = & \| \aa (H-\tau-i\ee)^{-1} P_c \aa \int_{0}^{\infty} ds e^{-is\tau} F(s) \|_{L_\tau^2} .\\
\end{eqnarray}
If we take the $L_x^2$-norm of both sides, we get
\begin{eqnarray}
  &      & \| \aa \int_{0}^{t} e^{i(t-s)(H+i \ee)} P_c \aa F(s) ds \|_{L_t^2L_x^2} \\
  & \lesssim & \| \aa (H-\tau+i\ee)^{-1} P_c \aa \int_{0}^{\infty} ds e^{-is\tau} F(s) \|_{L_\tau^2L_x^2} \\
  & \lesssim & \sup_{\tau} \| \aa (H-\tau+i\ee)^{-1} P_c \aa \|_{L^2 \rightarrow L^2 }  \| \int_{0}^{\infty} ds e^{-is\tau} F(s) \|_{L_\tau^2L_x^2} \\
  & \lesssim & \| F \|_{L_t^2L_x^2} \;\;\; \text{by Plancherel and Lemma \ref{l4}} .
\end{eqnarray}
Now sending $ \ee $ to 0, we have 
\begin{equation} \| \aa \int_{0}^{t} e^{i(t-s)H} P_c \aa F(s) ds \|_{L_t^2L_x^2} \lesssim \| F \|_{L_t^2L_x^2} \end{equation}
as needed.
\end{proof}

\begin{lemma} \label{l7} (Mixed Strichartz weighted estimates)
Let $(q,p)$ and $(\tilde{p},\tilde{q}) $ be Strichartz admissible. Then
\begin{eqnarray} 
  &          & \| \int_{0}^{t} e^{i(t-s)H} P_c F(s) ds \|_{L_t^q L_x^p} + \| \aa \int_{0}^{t} e^{i(t-s)H} P_c F(s) ds \|_{L_t^2 L_x^2} \\
  & \lesssim & \min ( \| \ll x \rr^{\sigma} F \|_{L_t^2L_x^2} , \| F \|_{L_t^{\tilde{q}'} L_x^{\tilde{p}'}} ) . 
\end{eqnarray}
\end{lemma}

\begin{proof}
First, 
\begin{eqnarray}
     \| \int_{0}^{\infty} e^{-isH} P_c F(s) ds \|_{L_x^2}^2
  =  (\int_{0}^{\infty} e^{-isH} P_c F(s) ds, \int_{0}^{\infty} e^{-itH} P_c F(s) dt).				
\end{eqnarray}
Moving the integrals through the inner product and rearranging the terms, we get
\begin{eqnarray}
 &   & \| \int_{0}^{\infty} e^{-isH} P_c F(s) ds \|_{L_x^2}^2 \\
 & = & \int_{0}^{\infty} ds ( P_c F(s), \int_{0}^{\infty} e^{-i(t-s)H} P_c F(s) dt) \\
 & = & \int_{0}^{\infty} ds (\invaa P_c F(s), \aa \int_{0}^{\infty} e^{-i(t-s)H} P_c F(s) dt) \\
 &   & \text{by H\"older inequality} \\
 & \leq & \| \invaa P_c F(s) \|_{L_s^2 L_x^2} \;\; \| \aa \int_{0}^{\infty} e^{-i(t-s)H} P_c F(s) dt \|_{L_s^2 L_x^2} \\
 &   & \text{and by lemma \ref{l6}} \\
 & \lesssim & \| \invaa P_c F(s) \|_{L_s^2 L_x^2}^2.
\end{eqnarray}

Hence,
 \begin{eqnarray}
   \| \int_{0}^{\infty} e^{i(t-s)H} P_c F(s) ds \|_{L_t^p L_x^q} 
   & = & \| e^{itH} \int_{0}^{\infty} e^{-isH} P_c F(s) ds \|_{L_t^p L_x^q} \\
   & \lesssim & \| \int_{0}^{\infty} e^{-isH} P_c F(s) ds \|_{L_x^2}
                  \;\;\; \text{by lemma \ref{l3}} \\
   & \lesssim & \| \invaa F(s) \|_{L_s^2 L_x^2} .
 \end{eqnarray}
Now, by a lemma of Christ-Kiselev (see \cite{MK}), we have
\begin{equation}
  \| \int_{0}^{t} e^{i(t-s)H} P_c F(s) ds \|_{L_t^p L_x^q} \lesssim \| \invaa F(s) \|_{L_s^2 L_x^2} .
\end{equation}  
 
Next, let $ \invaa g(x,t) \in L_t^2 L_x^2 $. Then
\begin{eqnarray}
  &   &  \int_0^{\infty} (\invaa g(x,t), \aa \int_{0}^{\infty} e^{i(t-s)H} P_c F(s) ds) dt \\
  & = & \int_0^{\infty} (g(x,t), \int_{0}^{\infty} e^{i(t-s)H} P_c F(s) ds) dt \\
\end{eqnarray}
Moving the integrals through the inner product and rearranging the terms, we get
\begin{eqnarray}
  &   & \int_0^{\infty} (\invaa g(x,t), \aa \int_{0}^{\infty} e^{i(t-s)H} P_c F(s) ds) dt \\
  & = & \int_0^{\infty} ds (\int_{0}^{\infty} e^{i(s-t)H} P_c g(x,t) dt, F(s)) \\
  &   & \text{by H\"older inequality} \\
  & \leq & \| \int_{0}^{\infty} e^{i(s-t)H} P_c g(x,t) dt \|_{L_t^q L_x^p} \| F(s) \|_{L_t^{q'} L_x^{p'}} \\
  & \lesssim & \| \invaa g \|_{L_x^2 L_t^2} \| F(s) \|_{L_t^{q'} L_x^{p'}}
\end{eqnarray}
Hence, 
\begin{equation} \| \aa \int_{0}^{\infty} e^{i(t-s)H} P_c F(s) ds \|_{L_t^2 L_x^2} \lesssim \| F(s) \|_{L_t^{q'} L_x^{p'}} . \end{equation}
Again, by the lemma of Christ-Kiselev, we have
\begin{equation} \| \aa \int_{0}^{t} e^{i(t-s)H} P_c F(s) ds \|_{L_t^2 L_x^2} \lesssim \| F(s) \|_{L_t^{q'} L_x^{p'}}  . \end{equation}
Now by lemma \ref{l3} and lemma \ref{l6}, we have shown lemma \ref{l7}.
\end{proof}

\begin{lemma} \label{l7-2a} (Derivative Strichartz estimates)
Let $p \geq 2$ and
let 
\begin{equation} 
  H_1 = H + K = - \Delta + 2 i A \cdot \nabla + i(\nabla \cdot A) + V + K 
\end{equation}
for a sufficiently large number $ K $. Then $ H_1 $ is a positive operator on $ L^p $, and
\begin{equation} \| \phi \|_{W^{1,p}} \sim \| H_1^{\frac{1}{2}} \phi \|_{L^p} . \end{equation}
From this, it follows that
\begin{equation} \| e^{-itH} f \|_{L_t^q W_x^{1,p}} \lesssim \| f \|_{H_x^1} \end{equation}
and
\begin{equation} 
  \| \int_{0}^{t} e^{i(t-s)H} P_c F(s) ds \|_{L_t^q W_x^{1,p}}
      \lesssim \| F \|_{L_t^{\tilde{q}'} W_x^{1,\tilde{p}'}} , 
\end{equation}
for Strichartz admissible $(q,p)$ and $(\tilde{p},\tilde{q}) $.
\end{lemma}
\begin{proof}
We would like to first show
\begin{equation} \| \phi \|_{W^{1,p}} \sim \| H_1^{\frac{1}{2}} \phi \|_{L^p} \;\; \text{for} \;\; \phi \in W^{1,p}. \end{equation}
Clearly $ \| \phi \|_{W^{0,p}} = \| \phi \|_{L^p} = \| H_1^{0} \phi \|_{L^p} $.
We will show in the appendix that for $ K $ large enough, $ H_1 $ is a positive operator on $ L^p $, and
\begin{equation} \| \phi \|_{W^{2,p}} \sim \| H_1 \phi \|_{L^p}  . \end{equation}

By Theorem 1 of \cite{MGC}, there exist positive numbers $ \ee $ and $ C $, such that $ H_1^{it} $ is a bounded operator on $ L^p $ for $ -\ee \leq t \leq \ee $ and $ \| H_1^{it} \| \leq C $.
Therefore the hypothesis of section 1.15.3 of \cite{HT} holds and we have that
\begin{equation} [D(H_1), D(H_1^0)]_{\frac{1}{2}} = D(H_1^{\frac{1}{2}}) . \end{equation}
Using that $ D(H_1) = W^{2,p}$, $ D(H_1^0) = L^p $ and $ [W^{2,p},L^p]_{\frac{1}{2}} = W^{1,p}$, we find that
\begin{equation} D(H_1^{\frac{1}{2}}) = W^{1,p} . \end{equation}
Now by section 1.15.2 of \cite{HT}, $H_1^{\frac{1}{2}} $ is an isomorphic mapping from $ D(H_1^{\frac{1}{2}}) = W^{1,p} $ onto $ L^p $. Therefore, we have
\begin{equation} \| \phi \|_{W^{1,p}} \sim \| H_1^{\frac{1}{2}} \phi \|_{L^p} . \end{equation}
Finally,
\begin{eqnarray}
  \| \int_{0}^{t} e^{i(t-s)H} P_c F(s) ds \|_{L_t^q W_x^{1,p}} 
  &    =     & \| \; \|  \int_{0}^{t} e^{i(t-s)H} P_c F(s) ds \|_{W_x^{1,p}} \|_{L_t^q} \\
  &   \sim   & \| \; \|  H_1^{\frac{1}{2}} \int_{0}^{t} e^{i(t-s)H} P_c F(s) ds \|_{L_x^p} \|_{L_t^q} \\
  &    =     & \| \; \| \int_{0}^{t} e^{i(t-s)H} P_c H_1^{\frac{1}{2}} F(s) ds \|_{L_x^p} \|_{L_t^q} \\
  & \lesssim & \| H_1^{\frac{1}{2}} F \|_{L_t^{\tilde{q}'} L_x^{\tilde{q}'}} \\
  &   \sim   & \| F \|_{L_t^{\tilde{q}'} W_x^{1,\tilde{p}'}}
\end{eqnarray}
\end{proof}
For $ s \in \mathbf{R} $, denote the norm $ \| \phi \|_{\ll x \rr^s L^2} $ by
\begin{equation} \| \phi \|_{\ll x \rr^s L^2} = \| \ll x \rr^{-s} \phi \|_{L^2} \end{equation}
and the norm $ \| \phi \|_{\ll x \rr^s H^1 } $ by
\begin{equation} \| \phi \|_{\ll x \rr^s H^1 } = \| \phi \|_{\ll x \rr^s L^2 } + \| \nabla \phi \|_{\ll x \rr^s L^2} . \end{equation}
Next we need derivative version of the weighted estimates of Lemma \ref{l6} - this is given in Lemma (\ref{l7-2b}) below. First, we need two preparatory lemmas.
\begin{lemma} \label{l7-2b-0}
For $ t > 0 $, let $ A_t(x) = \frac{1}{\sqrt{t}} A(\frac{x}{\sqrt{t}}) $ and $ V_t(x) = \frac{1}{t} V(\frac{x}{\sqrt{t}}) $. 
Let 
\begin{equation} 
  \tilde{H} =  - \Delta 
            +  2 i A_t \cdot \nabla + i(\nabla \cdot A_t) 
            + V_t + \frac{1}{t} K + 1 .
\end{equation}
Then there exists $ T > 0 $ such that $ \sup_{t > T} \| \tilde{H}^{-1} \|_{L^2 \rightarrow H^2} < \infty $. 
\end{lemma}
\begin{proof}
Take $ t \geq 1$. For $ \phi \in L^2 $, let $ h = \tilde{H}^{-1} \phi $. Then
\begin{eqnarray}
 \| \phi \|^2_2 
 & = & \Big( (- \Delta 
            + 2 i A_t \cdot \nabla + i( \nabla \cdot A_t) 
            + V_t + \frac{1}{t} K + 1)h, \\
 &   &
          \;\; (- \Delta 
            + 2 i A_t \cdot \nabla + i( \nabla \cdot A_t)  
            + V_t + \frac{1}{t} K + 1)h \Big) \\
 & = & \| \Delta h \|_2^2 + \| h \|_2^2 + \| A_t \cdot \nabla h \|_2^2 + 2 \| \nabla h \|_2^2  + F \\
 & \gtrsim & \| \Delta h \|_2^2 + \| h \|_2^2 + F \\
\end{eqnarray}
where $ F $ denotes the rest of the terms, and recall that $ q > 3 $.
We would like to show that every term in $ F $ is bounded by $ \| h \|^2_{H^2} $.
Here,
\begin{eqnarray} 
| F | 
& \leq & 2 \|  (\Delta h) (A_t \cdot \nabla h)  \|_1
      + 2 \| (\Delta h) ( \nabla \cdot A_t + V_t + \frac{1}{t} K) h \|_1 \\
&   &  + 2 \| [A_t (\nabla \cdot A_t + V_t + \frac{1}{t} K)] \cdot (\nabla h) h \|_1 \\
&   &  + 2 \| A_t \cdot (\nabla h) h \|_1
       + 2 \| (A_t+V_t+ \frac{1}{t} K)^2 h^2 \|_1         
\end{eqnarray}
Here,
\begin{eqnarray} 
\|  (\Delta h) (A_t \cdot \nabla h) \|_1
& \lesssim & \frac{1}{\sqrt{t}} \| \Delta h \|_2 \| (A(\frac{.}{\sqrt{t}}) \|_{L^\infty + L^q} (\| \nabla h \|_2 + \| \nabla h \|_{\frac{2q}{q-2}}) \\
&          & \text{where} \frac{2q}{q-2} < 6 \\
& \lesssim & \frac{1}{\sqrt{t}} \| \Delta h \|_2 \| (A(\frac{.}{\sqrt{t}}) \|_{L^\infty + L^q} 
                                   (\| \nabla h \|_2 + \| \Delta h \|_2^{\frac{3}{q}} \| \nabla h \|_2^{\frac{q-3}{q}} )   \\
& \lesssim & t^{-\frac{(q-3)}{2q} } \| \Delta h \|_2 \| A \|_{L^\infty + L^q} 
                                   (\| \nabla h \|_2 + \| \Delta h \|_2^{\frac{3}{q}} \| \nabla h \|_2^{\frac{q-3}{q}} ), \\
\end{eqnarray}
\begin{eqnarray} 
& & \| (\Delta h) ((\nabla \cdot A_t) + V_t + \frac{1}{t} K) h \|_1 \\
& \lesssim & \frac{1}{t} \| \Delta h \|_2 
             (\| (\nabla \cdot A) (\frac{.}{\sqrt{t}} ) \|_{L^\infty + L^2} + \| V(\frac{.}{\sqrt{t}}) \|_{L^\infty + L^2} + K) 
             (\| h \|_2 + \| h \|_\infty ) \\
& \lesssim & t^{\frac{-1}{4}} \| \Delta h \|_2 
             (\| \nabla \cdot A \|_{L^\infty + L^2} + \| V \|_{L^\infty + L^2} + K) 
             (\| h \|_2 + \| h \|_2^{\frac{1}{4}} \| \Delta h \|_2^{\frac{3}{4}} ) .\\
\end{eqnarray} 

Similar bounds hold for the other terms of $ F $.
We conclude that 
\begin{equation} \| \phi \|_2^2 \geq (1 + o(1)) \| h \|^2_{H^2} \;\; \text{as} \;\; t \rightarrow \infty. \end{equation}
Hence, for all $ t $ large enough, we have 
\begin{equation} \| h \|^2_{H^2} \lesssim \| \phi \|_2^2 . \end{equation}
\end{proof}
\begin{lemma} \label{l7-2b-1}
Let $ H_1 $ be as in lemma \ref{l7-2a}. For $ \phi \in L^2 $ and $t > 0$, we have
\begin{equation} \| \nabla (H_1+t)^{-1} \phi \|_{L^2} \lesssim (1+t)^{-\frac{1}{2}} \| \phi \|_{L^2}  . \end{equation}
\end{lemma}
\begin{proof}
For $ \phi \in L^2 $, let $ \psi = (H_1 + t)^{-1} \phi $. For $ t $ bounded away from zero, define $ \hat{\psi} $ by $ \psi(x) = \frac{1}{t} \hat{\psi}(\sqrt{t}x) $. Then $ \Delta \psi(x) = \Delta \hat{\psi}(\sqrt{t}x) $, $ \nabla \psi(x) = \frac{1}{\sqrt{t}} \nabla \hat{\psi}(\sqrt{t}x) $ and $ V(x) \psi(x) = \frac{1}{t} V(x) \hat{\psi}(\sqrt{t}x) $ and
\begin{equation} (\tilde{H} \hat{\psi})(\sqrt{t}x) = \phi(x) . \end{equation}
Replacing $ x $ by $ \frac{x}{\sqrt{t}}$ and inverting $ \tilde{H} $, we get
\begin{equation} \hat{\psi}(x) = \tilde{H}^{-1} \phi(\frac{x}{\sqrt{t}})  . \end{equation}
Hence,
\begin{equation} \psi(x) = \frac{1}{t}[\tilde{H}^{-1} \phi(\frac{.}{\sqrt{t}})](\sqrt{t}x) \end{equation}
and
\begin{equation} \nabla \psi(x) = \frac{1}{\sqrt{t}} [\nabla (\tilde{H})^{-1} \phi(\frac{.}{\sqrt{t}})](\sqrt{t}x). \end{equation}
By Lemma \ref{l7-2b-0}, $ \| \tilde{H}^{-1} \|_{L^2 \rightarrow L^2} $ is uniformly bounded for $ t \geq T$. Therefore,
\begin{eqnarray}
\| \nabla \psi(x) \|_2 
& = &  \| \frac{1}{\sqrt{t}} [\nabla \tilde{H}^{-1} \phi(\frac{.}{\sqrt{t}})](\sqrt{t}x) \|_2 \\
& = & t^{{\frac{-3}{4}} - \frac{1}{2}} \| \nabla \tilde{H}^{-1} \phi(\frac{.}{\sqrt{t}}) \|_2 \\
& \lesssim  & t^{{\frac{-3}{4}} - \frac{1}{2}} \| \nabla \tilde{H}^{-1} \|_{L^2 \rightarrow L^2}  \| \phi(\frac{.}{\sqrt{t}}) \|_2 \\
& = & t^{-\frac{1}{2}} \| \phi \|_2
\end{eqnarray}
Therefore, for $ t \geq T $,
\begin{equation} \| \nabla (H_1+t)^{-1} \phi \|_2 \lesssim t^{-\frac{1}{2}} \| \phi \|_2 \end{equation}
and the lemma follows.
\end{proof}
\begin{lemma} \label{l7-2b} (Derivative weighted estimates) 
Let $ H_1 $ be as in lemma \ref{l7-2a}. We have
\begin{equation} \| \phi \|_{\ll x \rr^{s} H^1} \sim \| H_1^{\frac{1}{2}} \phi \|_{\ll x \rr^{s} L^2}  \;\; \text{for} \;\; s \in \RR. \end{equation}
From this, it follows that
\begin{equation} 
  \| \aa \int_{0}^{t} e^{i(t-s)H} P_c F(s) ds \|_{L_t^2 H_x^1} 
      \lesssim  \| \ll x \rr^{\sigma} F \|_{L_t^2 H_x^1} . 
\end{equation}
\end{lemma}
\begin{proof}
Since $ \| f \|_{\ll x \rr^{s} H^1} = \| \ll x \rr^{-s} f \|_{L^2} + \| \nabla \ll x \rr^{-s} f \|_{L^2} $, to show the lemma, it suffices to show
\begin{equation} \| \ll x \rr^{-s} H_1^{-\frac{1}{2}} \ll x \rr^{s} \|_{L^2 \rightarrow L^2} < \infty \end{equation}
and
\begin{equation} \| \nabla \ll x \rr^{-s} H_1^{-\frac{1}{2}} \ll x \rr^{s} \|_{L^2 \rightarrow L^2} < \infty . \end{equation}
The second bound above is the harder of the two. We will show the second bound and the first one follows by a similar argument.
First,
\begin{equation} 
   \nabla \ll x \rr^{-s} H_1^{-\frac{1}{2}} \ll x \rr^{s} \phi 
   =  \nabla H_1^{-\frac{1}{2}} \phi + \nabla \ll x \rr^{-s} [H_1^{-\frac{1}{2}}, \ll x \rr^{s}] \phi
\end{equation}
Now $ \nabla H_1^{-\frac{1}{2}} $ is bounded from $ L^2 $ to $ L^2 $ since $ H_1^{-\frac{1}{2}} $ maps from $ L^2 $ to $ H^1 $ while $ \nabla $ maps from $ H^1 $ to $ L^2 $.

For the second term, we use $ H_1^{-\frac{1}{2}} = \int_0^{\infty} \frac{dt}{\sqrt{t}} (H_1+t)^{-1} $ and $ [(H_1+t)^{-1}, \ll x \rr^s] = (H_1+t)^{-1} [H_1+t,\ll x \rr^s] (H_1+t)^{-1} $ to get
\begin{equation} 
  \nabla \ll x \rr^{-s} [H_1^{-\frac{1}{2}}, \ll x \rr^{s}]
  = \nabla \ll x \rr^{-s} \int_{0}^{\infty} \frac{dt}{\sqrt{t}} (H_1+t)^{-1} [H_1+t,\ll x \rr^s] (H_1+t)^{-1}
\end{equation}
Recall that
\begin{equation} H_1 = - \Delta + 2iA \cdot \nabla + i(\nabla \cdot A) + V + K , \end{equation}
so 
\begin{eqnarray}
 [H_1+t,\ll x \rr^s] 
 & = & (-\Delta \ll x \rr^s) -2 (\nabla \ll x \rr^s) \cdot \nabla + 2i A \cdot (\nabla \ll x \rr^s).
\end{eqnarray}
Let $ g(x) = (-\Delta \ll x \rr^s) + 2i A \cdot (\nabla \ll x \rr^s) $ and $ h(x) = -2(\nabla \ll x \rr^s) $. Then 
\begin{equation}
  \nabla \ll x \rr^{-s} [H_1^{-\frac{1}{2}}, \ll x \rr^{s}] 
  = \nabla \ll x \rr^{-s} \int_{0}^{\infty} \frac{dt}{\sqrt{t}} (H_1+t)^{-1} (g(x) + h(x) \cdot \nabla)(H_1+t)^{-1}  
 .\end{equation}
Since $ g(x) \lesssim \ll x \rr^{s-1} $, we rewrite the $ g(x)$-part of the above as
\begin{eqnarray}
 & & \nabla \ll x \rr^{-s} \int_{0}^{\infty} \frac{dt}{\sqrt{t}} (H_1+t)^{-1} g(x) (H_1+t)^{-1} \\
 & = &  \nabla \int_{0}^{\infty} \frac{dt}{\sqrt{t}} \ll x \rr^{-s} g(x) (H_1+t)^{-1} (H_1+t)^{-1} \\
 &  & + \nabla \ll x \rr^{-s} \int_{0}^{\infty} \frac{dt}{\sqrt{t}} (H_1+t)^{-1} [H_1+t,g(x)] (H_1+t)^{-1} (H_1+t)^{-1}
\end{eqnarray}
The first part of the above sum is bounded. For the second part, writing $ [H_1+t,g(x)] = \tilde{g}(x) + \tilde{h}(x) \cdot \nabla $ as before , we can iterate the above process until $ \tilde{g}(x) \lesssim 1$.
Since $ h(x) \lesssim \ll x \rr^{s-1} $, so by the similar argument, we have 
\begin{eqnarray}
 & & \nabla \ll x \rr^{-s} \int_{0}^{\infty} \frac{dt}{\sqrt{t}} (H_1+t)^{-1} h(x) \cdot \nabla (H_1+t)^{-1} \\
 & = &  \nabla \int_{0}^{\infty} \frac{dt}{\sqrt{t}} \ll x \rr^{-s} h(x) (H_1+t)^{-1} \nabla (H_1+t)^{-1} \\
 &  & + \nabla \ll x \rr^{-s} \int_{0}^{\infty} \frac{dt}{\sqrt{t}} (H_1+t)^{-1} [H_1+t,h(x)] (H_1+t)^{-1} \nabla (H_1+t)^{-1}
\end{eqnarray}
As before, the first part of the above sum is bounded. For the second part, $ [H_1+t,g(x)] = \tilde{g}(x) + \tilde{h}(x) \cdot \nabla $ as before , we can iterate the above process until $ \tilde{h}(x) \lesssim 1$. As a result, it suffices to consider
\begin{equation} \int_{0}^{\infty} \frac{dt}{\sqrt{t}} ((H_1+t)^{-1})^{m} \end{equation}
and 
\begin{equation} \int_{0}^{\infty} \frac{dt}{\sqrt{t}} ((H_1+t)^{-1} \nabla (H_1+t)^{-1})^{m} \end{equation}
for $ m \geq 1 $. 
Now by lemma \ref{l7-2b-1}, both of the expressions above are bounded in $L^2$.
\end{proof}
Now, to prove theorem \ref{t7-2}, apply lemma \ref{l7-2a} and \ref{l7-2b} to lemma \ref{l7}, we get the result.

Finally, we need a lemma from \cite{GNT} for the projection operator $P_c$ onto the continuous spectral subspace.
\begin{lemma} \label{l8} (Continuous spectral subspace comparison)
Let the continuous spectral subspace $ \mathcal{H}_c[z] $ be defined as 
\begin{equation} \mathcal{H}_c[z] = \{ \eta \in L^2 | \ll i \eta, D_1 Q[z] \rr = \ll i \eta, D_2 Q[z] \rr = 0 \} . \end{equation}
Then there exists $ \delta > 0$ such that for each $ z \in \CC$ with $ |z| \leq \delta $, there is a bijective operator $ R[z] : \text{Ran} \; P_c 
\rightarrow \mathcal{H}_c[z]$ satisfying
\begin{equation} P_c |_{\mathcal{H}_c[z]} = (R[z])^{-1} . \end{equation}
Moreover, $ R[z] - I$ is compact and continuous in $z$ in the operator norm on any space $Y$ satisfying $ H^2 \cap W^{1,1} 
\subset Y \subset 
H^{-2} + L^{\infty}$.
\end{lemma}
The proof of lemma \ref{l8} is given in lemma 2.2 of \cite{GNT}. We will use lemma \ref{l8} with $Y = L^2$.

\section{Proof of the main theorem} \label{secMain}
Lemma \ref{l1} gives the following corollary which will form part of the main theorem. 
\begin{lemma} \label{l1c1} (Best decomposition)
There exists $ \delta > 0 $ such that any $ \psi \in H^1 $ satisfying $ \| \psi \|_{H^1} \leq \delta $ can be uniquely decomposed as
\begin{equation} \psi = Q[z] + \eta \end{equation}
where $ z \in \CC $, $ \eta \in H^1 $, $ \ll i \eta, D_1 Q[z] \rr = \ll i \eta, D_2Q[z] \rr = 0 $, and $ |z| + \| \eta \|_{H^1} \lesssim \| \psi \|_{H^1} $.
\end{lemma}
The proof of lemma \ref{l1c1} is essentially an application of the implicit function theorem on the equation $ B(z) = 0 $ with
\begin{equation} B(z) = (B_1(z), B_2(z)), \;\; B_j = \ll i(\psi - Q[z]), D_j Q[z] \rr \;\; \text{for} \;\; j = 1,2  .\end{equation}
Details can be found in lemma 2.3 of \cite{GNT}.

Now, we prove theorem \ref{t1}.
\begin{proof}
Substitute
\begin{equation} \psi(t) = Q[z(t)] + \eta(t) \end{equation}
into equation \eqref{E1} to get
\begin{eqnarray}
i(DQ \dot{z} + \pp_t \eta) & = & HQ + H \eta + g(Q+\eta)
\end{eqnarray}
where for $ w \in \CC $, we denote $ DQ[z]w = D_1 Q[z] \Re w + D_2 Q[z] \Im w $.
Since $ HQ + g(Q) = EQ $ and $ DQ[z]iz = iQ[z] $ (since $ Q[e^{i \alpha} z] = e^{i \alpha} Q[z] $ for $ \alpha \in \RR $), we have
\begin{eqnarray}
i \pp_t \eta & = &  H \eta - i DQ \dot{z} + EQ - g(Q) + g(Q + \eta) \\
             & = &  H \eta - i DQ(\dot{z}+iEz) - g(Q) + g(Q + \eta).
\end{eqnarray}
We can write this as 
\begin{equation} i \pp_t \eta  = H \eta + F \end{equation}
where
\begin{equation} F = g(Q + \eta) - g(Q) - i DQ(\dot{z} + iEz) .\end{equation}
In integral form,
\begin{equation} \eta(t) = e^{-itH}(\eta(0) - i \int_0^t e^{isH} F(s) ds) .\end{equation}
Let $ \eta_c = P_c \eta $. Then 
\begin{equation} \eta_c = e^{-itH} P_c \eta(0) - i \int_0^t e^{i(t-s)H} P_c F(s) ds  .\end{equation}
Then for fixed $ \sigma > 4$, since $ \eta = \Re [z] \eta_c $, we have
 \begin{eqnarray}
 \| \eta \|_X
 & \lesssim & \| \eta_c \|_X \\
 & \lesssim & \| \eta(0) \|_{H_x^1} 
                  + \| \int_0^t e^{-i(s-t)H} P_c (F(s) - 2Q |\eta|^2 - \bar{Q} \eta^2 -|\eta|^2 \eta) ds\|_X \\ 
 &          &     + \| \int_0^t e^{-i(s-t)H} P_c (2 Q |\eta|^2 + \bar{Q} \eta^2 + |\eta|^2 \eta) ds \|_X \\
 & \lesssim & \| \eta(0) \|_{H_x^1} + \| \int_0^t e^{-i(s-t)H} P_c (F(s) - 2Q |\eta|^2 - \bar{Q} \eta^2 -|\eta|^2 \eta) ds\|_X \\
 &          &                         + \| Q \eta^2 \|_{L_t^{\frac{3}{2}} W_x^{1,\frac{18}{13}}}
                                      + \| \eta^3 \|_{L_t^{\frac{3}{2}} W_x^{1,\frac{18}{13}}} .\\
 \end{eqnarray}
For $ \| Q \eta^2 \|_{L_t^{\frac{3}{2}} W_x^{1,\frac{18}{13}}} $, we have
\begin{eqnarray}
  \| Q \eta^2 \|_{L_t^{\frac{3}{2}} W_x^{1,\frac{18}{13}}}
  &    =     &  \| Q \eta^2 \|_{L_t^{\frac{3}{2}} L_x^{\frac{18}{13}}} + \| \nabla (Q \eta^2) \|_{L_t^{\frac{3}{2}} L_x^{\frac{18}{13}}} \\
  & \lesssim &  \| (|Q| + |\nabla Q|) \eta^2 \|_{L_t^{\frac{3}{2}} L_x^{\frac{18}{13}}} 
                + \| Q \eta \nabla \eta \|_{L_t^{\frac{3}{2}} L_x^{\frac{18}{13}}} \\
  & \lesssim & \| Q \|_{L_t^\infty W_x^{1,6}} \| \eta \|^2_{L_t^{3} L_x^{\frac{18}{5}}} 
                + \| Q \|_{L_t^\infty L_x^6} \| \eta \|_{L_t^{3} L_x^{\frac{18}{5}}} \| \nabla \eta \|_{L_t^{3} L_x^{\frac{18}{5}}} \\
  & \lesssim & \| Q \|_{W_x^{1,6}} \| \eta \|_X^2  .         
\end{eqnarray}
For $ \| \eta^3 \|_{L_t^{\frac{3}{2}} W_x^{1,\frac{18}{13}}} $, we have
\begin{eqnarray}
  \| \eta^3 \|_{L_t^{\frac{3}{2}} W_x^{1,\frac{18}{13}}}
  &    =    &  \| \eta^3 \|_{L_t^{\frac{3}{2}} L_x^{\frac{18}{13}}} + \| \nabla \eta^3 \|_{L_t^{\frac{3}{2}} L_x^{\frac{18}{13}}} \\
  & \lesssim &  \| \eta^3 \|_{L_t^{\frac{3}{2}} L_x^{\frac{18}{13}}} 
                + \| \eta^2 \nabla \eta \|_{L_t^{\frac{3}{2}} L_x^{\frac{18}{13}}} \\
  & \leq    &  \| \eta^2 \|_{L_t^{3} L_x^{\frac{9}{4}}} \| \eta \|_{L_t^3 L_x^{\frac{18}{5}}} 
                + \| \eta^2 \|_{L_t^{3} L_x^{\frac{9}{4}}}  \| \nabla \eta \|_{L_t^{3} L_x^{\frac{18}{5}}} \\
  & \leq    &  \| \eta \|_{L_t^6 L_x^{\frac{9}{2}}}^2 \| \eta \|_{L_t^3 W_x^{1,\frac{18}{5}}} .
\end{eqnarray}
Now, using $ \| \eta \|_{L_x^{\frac{9}{2}}} \lesssim \| \nabla \eta \|^{\frac{1}{2}}_{L_x^2} \| \eta \|^{\frac{1}{2}}_{L_x^{\frac{18}{5}}} $, we get
\begin{equation} \| \eta \|_{L_t^{6} L_x^{\frac{9}{2}}} \lesssim \| \nabla \eta \|_{L_t^\infty L_x^2}^{\frac{1}{2}} \| \eta \|_{L_t^3 L_x^{\frac{18}{5}}}^{\frac{1}{2}}  .\end{equation}
So
\begin{equation} 
  \| \eta^3 \|_{L_t^{\frac{3}{2}} W_x^{1,\frac{18}{13}}}
  \lesssim \| \nabla \eta \|_{L_t^\infty L_x^2} \| \eta \|_{L_t^3 W_x^{1,\frac{18}{5}}}^2 \lesssim \| \eta \|_X^3 . 
\end{equation}
Together we have
\begin{eqnarray}
  \| \eta \|_X
  & \lesssim & \| \eta(0) \|_{H_x^1} + \| \int_0^t e^{-i(s-t)H} P_c (F(s) - 2 Q |\eta|^2 - \bar{Q} \eta^2 -|\eta|^2 \eta) ds\|_X \\
  &          &                         + \| Q \|_{W_x^{1,6}} \| \eta \|_X^2 + \| \eta \|_X^3 \\
  & \lesssim & \| \eta(0) \|_{H_x^1} + \| (F(s) - 2 Q | \eta |^2 - \bar{Q} \eta^2 - |\eta|^2 \eta) \|_{L_t^2 \aa H_x^1} \\
  &          &    + \| \eta \|_X^2 + \| \eta \|_X^3
\end{eqnarray}

Next, for $ g(\psi) = |\psi|^2 \psi $,
 \begin{eqnarray}
 &          & \| (F - 2 Q |\eta|^2 - \bar{Q} \eta^2 - |\eta|^2 \eta) \|_{L_t^2 \aa H_x^1} \\
 & =        & \| Q^2 \bar{\eta} + 2 |Q|^2 \eta - i DQ(\dot{z} + iEz) \|_{L_t^2 \aa H_x^1} \\
 & \lesssim & \| \ll x \rr^{2 \sigma} Q^2 \|_{W_x^{1,\infty}} \| \eta \|_{L_t^2 \invaa H_x^1} 
               + \| DQ \|_{\aa H_x^1} \| \dot{z} + iEz \|_{L_t^2} .
 \end{eqnarray}
Next, we would like to bound $ (\dot z + i Ez) $. Recall that we imposed
\begin{equation} \ll i \eta, \frac{\pp}{\pp z_1} Q[z] \rr = 0 \;\;\; \text{and} \;\;\; \ll i \eta, \frac{\pp}{\pp z_2} Q[z] \rr = 0 \end{equation}
through Lemma \ref{l1c1}. By Gauge covariance of $ Q $, we have
\begin{equation} Q[e^{i \alpha} z] = e^{i \alpha} Q[z] .\end{equation}
So for $ z = z_1 + i z_2 $,
\begin{equation} Q[z] = e^{i \alpha} \tilde{Q}[|z|^2] \;\;\; \text{where} \;\; \alpha = \tan^{-1}(\frac{z_2}{z_1}) .\end{equation}
Here $ \tilde{Q}: \RR^+ \rightarrow \RR $. So
\begin{equation} \pp_{z_1}Q = \pp_{z_1}(e^{i\alpha}) \tilde{Q} + 2 z_1 e^{i \alpha} \tilde{Q}' 
              = e^{i\alpha} i (\pp_{z_1}\alpha) \tilde{Q} + 2 z_1 e^{i \alpha} \tilde{Q}' 
              = i (\pp_{z_1}\alpha) Q + 2 z_1 e^{i \alpha} \tilde{Q}' \end{equation}
and
\begin{equation} \pp_{z_2}Q = \pp_{z_2}(e^{i\alpha}) \tilde{Q} + 2 z_2 e^{i \alpha} \tilde{Q}' 
              = e^{i\alpha} i (\pp_{z_2}\alpha) \tilde{Q} + 2 z_2 e^{i \alpha} \tilde{Q}' 
              = i (\pp_{z_2}\alpha) Q + 2 z_2 e^{i \alpha} \tilde{Q}' .\end{equation}
So
\begin{eqnarray} 
  0 & = & \ll i \eta, -z_2 \pp_{z_1}Q + z_1 \pp_{z_2}Q \rr = \ll \eta, -z_2 (\pp_{z_1}\alpha) Q + z_1 (\pp_{z_2}\alpha) Q \rr \\
    & = & (-z_2 (\pp_{z_1}\alpha) + z_1 (\pp_{z_2}\alpha))  \ll \eta, Q \rr = \ll \eta, Q \rr.
\end{eqnarray}
Now differentiate $ \ll i \eta, \frac{\pp}{\pp z_1} Q[z] \rr = 0 $ and $ \ll i \eta, \frac{\pp}{\pp z_2} Q[z] \rr = 0 $ with respect to $t$ and substitute $ i \pp_t \eta  = H \eta + F $, we get
\begin{eqnarray}
 0 & = & \ll i \pp_t \eta, \frac{\pp}{\pp z_j} Q[z] \rr + \ll i \eta,  D  \frac{\pp}{\pp z_j} Q \dot{z} \rr \\
   & = & \ll H \eta + F, \frac{\pp}{\pp z_j} Q[z] \rr + \ll i \eta,  D  \frac{\pp}{\pp z_j} Q \dot{z} \rr \\
\end{eqnarray}
Recall that $ F = g(Q + \eta) - g(Q) - i DQ(\dot{z} + iEz) $. Therefore, we have
\begin{eqnarray}
 0 & = & \ll  H \eta + g(Q + \eta) - g(Q) - i DQ(\dot{z} + iEz), \frac{\pp}{\pp z_j} Q[z] \rr + \ll i \eta,  D  \frac{\pp}{\pp z_j} Q \dot{z} \rr \\
   & = &  \ll (H \eta + \frac{\pp}{\pp \ee} g(Q + \ee \eta)|_{\ee = 0} ) + (g(Q + \eta) - g(Q) - \pp_{\ee}^0 g(Q + \ee \eta) ) \\
   &   & - i DQ(\dot{z} + iEz), \frac{\pp}{\pp z_j} Q[z] \rr 
          + \ll i \eta,  D  \frac{\pp}{\pp z_j} Q \dot{z} \rr \\
\end{eqnarray}
From the above, we get that
\begin{eqnarray}
 & & \ll (g(Q + \eta) - g(Q) - \pp_{\ee}^0 g(Q + \ee \eta) ), \frac{\pp}{\pp z_j} Q[z] \rr \\
 & = & \ll - i DQ(\dot{z} + iEz), \frac{\pp}{\pp z_j} Q[z] \rr \\
 & & + \ll (H \eta + \pp_{\ee}^0 g(Q + \ee \eta) ), \frac{\pp}{\pp z_j} Q[z] \rr \\
 & & + \ll i \eta,  D  \frac{\pp}{\pp z_j} Q \dot{z} \rr .\\
\end{eqnarray}
Let $ \HH \eta = H \eta + \pp_{\ee}^0 g(Q + \ee \eta)$. By the symmetry of $ \HH$ and differentiating equation (\ref{eNEP}) by $ z_j $, we have
\begin{eqnarray}
 \ll \HH \eta,\frac{\pp}{\pp z_j} Q \rr = \ll \eta, \HH \frac{\pp}{\pp z_j} Q \rr 
   & = & \ll \eta, E  \frac{\pp}{\pp z_j} Q \rr + (\frac{\pp}{\pp z_j} E) \ll \eta, Q \rr \\
   & = & \ll \eta, E  \frac{\pp}{\pp z_j} Q \rr = \ll i \eta, i E  \frac{\pp}{\pp z_j} Q \rr \\
   & = & \ll i \eta, E  \frac{\pp}{\pp z_j} DQiz \rr
\end{eqnarray}
using $  \ll \eta, Q \rr = 0 $ and $ DQ[z]iz = iQ[z]$.
So 
\begin{eqnarray}
& & \ll (g(Q + \eta) - g(Q) - \pp_{\ee}^0 g(Q + \ee \eta) ), \frac{\pp}{\pp z_j} Q[z] \rr \\ 
& = &  \ll - i DQ(\dot{z} + iEz), \frac{\pp}{\pp z_j} Q[z] \rr
 + \ll i \eta, E  \frac{\pp}{\pp z_j} DQiz \rr + \ll i \eta,  D  \frac{\pp}{\pp z_j} Q \dot{z} \rr \\
& = &  \ll - i DQ(\dot{z} + iEz), \frac{\pp}{\pp z_j} Q[z] \rr
 + \ll i \eta, (D \frac{\pp}{\pp z_j}Q) (\dot{z}+iEz) \rr \\
\end{eqnarray}

For $ g(\psi) = | \psi |^2 \psi $, 
\begin{equation} \pp_{\ee}^0 g(Q + \ee \eta) = Q^2 \bar{\eta} + 2 |Q|^2 \eta .\end{equation}
Therefore,
\begin{eqnarray} 
  g(Q + \eta) - g(Q) - \pp_{\ee}^0 g(Q + \ee \eta) & = & | Q + \eta |^2 (Q + \eta) -|Q|^2 Q - Q^2 \bar{\eta} - 2 |Q|^2 \eta \\ 
                                                   & = & 2 Q |\eta|^2 + \bar{Q} \eta^2 + |\eta|^2 \eta \\
\end{eqnarray}

Since
\begin{equation} \ll \frac{\pp}{\pp z_j} Q, i \frac{\pp}{\pp z_k} Q \rr = j - k + o(1), \end{equation} 
we have that
\begin{equation} |\dot{z} + iEz| \lesssim | \ll 2 Q |\eta|^2 + \bar{Q} \eta^2 + |\eta|^2 \eta, DQ \rr | (1 + \| \eta \|_{L^2}) .\end{equation}
Therefore,
\begin{eqnarray}
&          & \| \dot{z} + iEz \|_{L_t^2} \\
& \lesssim & \| \ll 2 Q |\eta|^2 + \bar{Q} \eta^2 + |\eta|^2 \eta, DQ \rr \|_{L_t^2} (1 + \| \eta \|_{L_t^{\infty} L_x^2}) \\
& \lesssim & ( \| Q DQ |\eta|^2 \|_{L_t^2 L_x^1} + \| DQ |\eta|^2 \eta \|_{L_t^2 L_x^1}) (1 + \| \eta \|_{L_t^{\infty} L_x^2})\\ 
& \leq     & ( \| Q DQ \|_{L_t^{\infty} L_x^2} \| \eta \|_{L_t^4 L_x^4}^2  
             + \| DQ \|_{L_t^{\infty} L_x^4} \| \eta \|_{L_t^6 L_x^4}^3 ) (1 + \| \eta \|_{L_t^{\infty} L_x^2}) \\ 
& \leq     & ( \| Q DQ \|_{L_t^{\infty} L_x^2} \| \eta \|_{L_t^{\infty} H^1}^{\frac{1}{2}} \| \eta \|_{L_t^3 
L_x^{\frac{18}{5}}}^{\frac{3}{2}}
             + \| DQ \|_{L_t^{\infty} L_x^4} \| \eta \|_{L_t^{\infty} H^1}^{\frac{5}{4}} \| \eta \|_{L_t^{\frac{7}{2}} 
L_x^{\frac{42}{13}}}^{\frac{7}{4}} ) \\
&          &     (1 + \| \eta \|_{L_t^{\infty} L_x^2}) \\ 
& \lesssim & \| \eta \|_X^2 + \| \eta \|_X^4          
\end{eqnarray}
For $ \| \eta \|_{L_t^4 L_x^4} $, we used
\begin{equation}  \| \eta \|_{L_x^4} \lesssim \| \nabla \eta \|_{L_x^2}^{\frac{1}{4}}  \| \eta \|_{L_x^{\frac{18}{5}}}^{\frac{3}{4}} .\end{equation}
For $ \| \eta \|_{L_t^6 L_x^4} $, we used
\begin{equation}  \| \eta \|_{L_x^4} \lesssim \| \nabla \eta \|_{L_x^2}^{\frac{5}{12}}  \| \eta \|_{L_x^{\frac{42}{13}}}^{\frac{7}{12}} .\end{equation}

Putting the preceding estimates together we have
\begin{equation} \label{R1} 
 \| \eta \|_X \lesssim \| \eta(0) \|_{H^1} + \| \ll x \rr^{2 \sigma} Q^2 \|_{L_x^{\infty}} \| \eta \|_X + \| \eta \|_X^2 + \| \eta \|_X^4 ,
\end{equation}
and since $ \| \ll x \rr^{2 \sigma} Q^2 \|_{L_x^{\infty}} << 1 $,
\begin{equation} \| \eta \|_X \leq C [\| \eta(0) \|_{H^1} + \| \eta \|_X^2 + \| \eta \|_X^4 ] \end{equation}
for some constant $ C \geq 1 $.

Now, let $ X_T $ be the norm defined by
\begin{eqnarray}
  \| \psi \|_{X_T}  
  & = & \| \aa \psi \|_{L_t^2([0,T], H_x^1)} + \| \psi \|_{L_t^{3}([0,T],W_x^{1,\frac{18}{5}}) } + \| \psi \|_{L_t^{\infty} ([0,T],H_x^1)}
\end{eqnarray}
Fix the initial condition $ \| \psi(0) \|_X $ to be small enough so that
\begin{equation} \| \eta(0) \|_{H^1} \leq \frac{1}{20C^2} .\end{equation} 
Let
\begin{equation} T_1 = \sup \{ T>0 : \| \eta \|_{X_T} \leq \frac{1}{10C} \} > 0 .\end{equation}
Then for $ 0 \leq T \leq T_1 $,
\begin{equation} \label{R2} 
 \| \eta \|_{X_T} \leq \frac{1}{20C} + \frac{1}{10^2C} + \frac{1}{10^4C^3} \leq \frac{1}{15C},
\end{equation}
showing that $ T_1 = \infty $.

Next, we would like to bound $\| \dot{z} + iEz \|_{L_t^1}$. We have
\begin{eqnarray}
&          &  \| \dot{z} + iEz \|_{L_t^1} \\
& \lesssim & \| \ll 2 Q |\eta|^2 + \bar{Q} \eta^2 + |\eta|^2 \eta, DQ \rr (1 + \| \eta \|_{L_x^2}) \|_{L_t^1} \\
& \lesssim & ( \| Q DQ |\eta|^2 \|_{L_t^1 L_x^1} + \| DQ |\eta|^2 \eta \|_{L_t^1 L_x^1}) (1 + \| \eta \|_{L_t^\infty L_x^2}) \\ 
&   \leq   & ( \| \invsqaa Q DQ \|_{L_t^{\infty} L_x^{\infty}} \| \ll x \rr^{-2\sigma} \eta^2 \|_{L_t^1 L_x^1} 
                + \| \invaa DQ \|_{L_t^{\infty} L_x^{\infty}} \| \aa  \eta^3 \|_{L_t^1 L_x^1} ) \\
&          &    (1 + \| \eta \|_{L_t^\infty L_x^2}) 
\end{eqnarray}
Here, the factor $\| \aa  \eta^3 \|_{L_t^1 L_x^1}$ can be bounded by 
\begin{eqnarray}
  \| \aa  \eta^3 \|_{L_t^1 L_x^1} 
  \leq  \| \aa  \eta \|_{L_t^2 L_x^2} \| \eta \|_{L_t^4 L_x^4}^2) 
  \lesssim \| \aa \eta \|_{L_t^2 L_x^2} \| \eta \|_{L_t^{\infty} H^1}^{\frac{1}{2}} \| \eta \|_{L_t^3 L_x^{\frac{18}{5}}}^{\frac{3}{2}}.
\end{eqnarray}
Putting everything together, we have
\begin{eqnarray}
 \| \dot{z} + iEz \|_{L_t^1} \lesssim \| \eta \|_X^2 + \| \eta \|_X^4
\end{eqnarray}

Therefore, $ | \partial_t (e^{i \int_0^t E(s) ds} z(t)) |  = | \dot{z} + iEz | \in L_t^1 $. This means that $ \lim_{t \rightarrow \infty} e^{i \int_0^t E(s) ds} z(t) $ exists. Since $ | e^{i \int_0^t E(s) ds} z(t) | = |z| $, $ \lim_{t \rightarrow \infty} |z(t)| $ exists. Furthermore, $ E $ is continuous and $ E(z) = E(|z|)$, so $ \lim_{t \rightarrow \infty} E(z(t)) $ exists.

Finally, let $ H = - \Delta + 2i A \cdotp \nabla + i(\nabla \cdotp A) + V$. So
\begin{equation} \eta_c(t) = e^{itH} (\eta_c(0) - i \int_0^t e^{-isH} P_c F(s) ds).\end{equation}
By Strichartz estimates as above, we have
\begin{equation} \| \int_S^T e^{-isH} P_c F(s) ds \|_{H^1} \lesssim \| F \|_X \rightarrow 0 \end{equation}
as $ T > S \rightarrow \infty$.
Therefore, $ \int_0^{\infty} e^{-isH} P_c F ds$ converges in $ H^1 $, and
\begin{equation} \lim_{t \rightarrow \infty} e^{-itH} \eta_c(t) = \eta_c(0) - i \int_0^{\infty} e^{-isH} P_c F(s) ds =: \eta_+ \end{equation}
for some $ \eta_+ \in H^1 $.
From this, we get that $ \eta_c(t) $ converges to 0 weakly in $ H_1$. Now, by compactness of $ R[z(t)] - I$, we have that $ \eta_d(t) := \eta(t) - \eta_c(t) = (R[z(t)] - I) \eta_c(t)$ converges to 0 strongly in $ H^1$. Therefore
\begin{equation} \| \eta(t) - e^{itH} \eta_+ \|_{H^1} \rightarrow 0 .\end{equation}
\end{proof}
\appendix
\section{Nonlinear bound states}
The following is the proof for Lemma \ref{l1}, the existence and exponential decay of nonlinear bound states. \\
%\begin{proof} (Existence of nonlinear bound states)
\textit{Proof of existence of nonlinear bound states:}

For each small $ z \in \CC $, we look for a solution 
\begin{equation} Q = z \phi_0 + q \;\; \text{and} \;\; E = e_0 + e' \end{equation}
of  
\begin{equation} \label{l1e1}
 (- \Delta + 2 i A \cdotp \nabla + i(\nabla \cdotp A) + V)Q + g(Q) = EQ 
\end{equation}
with $ (\phi_0, q) = 0 $ and $ e' \in \RR $ small.

Let $ H_0 = - \Delta + 2i A \cdotp \nabla + i(\nabla \cdotp A) + V - e_0 $. If we substitute $ Q = z \phi_0 + q $ and $ E = e_0 + e' $ into equation (\ref{l1e1}), we get
\begin{equation} \label{l1e2}
 H_0 q + g(z \phi_0 + q) = e'(z \phi_0) + e' q .
\end{equation}

Projecting equation (\ref{l1e2}) on the $ \phi_0 $ and $ \phi_0^{\perp}$ directions, we get
\begin{equation} e' z = (\phi_0, g(z \phi_0 + q)) \end{equation}
and
\begin{equation} H_0 q = - P_c g(z \phi_0 + q) + e'q.\end{equation}

Now, let
\begin{equation} K = \{ (q,e') \in H_{\perp}^2 \times \RR | \| q \|_{H^2} \leq |z|^2, |e'| \leq |z| \} \end{equation}
for sufficiently small $ z \in \CC $ where $ H_{\perp}^2 = \{q \in H^2| (q, \phi_0) = 0 \} $. Also, define the map $ M: (q_0, e'_0) \mapsto (q_1, e'_1) $ by
\begin{equation} g_0 := g(z \phi_0 + q_0) ,\end{equation}
\begin{equation} z e'_1 := (\phi_0, g_0) \end{equation}
and
\begin{equation} q_1 := H_0^{-1}(-P_c g_0 + e'_0 q_0) .\end{equation} 

Now if $ (q_0, e'_0) \in K $, we have
\begin{equation} |ze'_1| = |(\phi_0,g_0)| = |(\phi_0, g(z\phi_0+q_0))| = |(\phi_0, |z \phi_0 + q_0 |^2 (z \phi_0 + q_0))| \lesssim O(z^3)\end{equation}
and
\begin{equation} \| q_1 \|_{H^2} \lesssim \| -P_c g_0 + e'_0 q_0 \|_{L^2} \leq \| g_0 \|_{L^2} + |e'_0| \| q_0 \|_{H^2} \lesssim O(z^3) .\end{equation}
Therefore, $ | e_1'| \lesssim O(z^2) $ and $ \| q_1 \|_{H^2} \lesssim O(z^3) $. 
This shows that $ M $ maps $ K $ into $ K $ for sufficiently small $z$.

Next, we would like to show that $ M $ is a contraction mapping.
Let $ (a_1, b_1) := M(q_0, e'_0) $ and $ (a_2, b_2) := M(q_1, e'_1) $ with $ g_j = g(z \phi_0 + q_j)$ for $ j = 0,1$. Then
\begin{eqnarray}
 |z(b_2 - b_1)| & = & |(\phi_0, g_0-g_1)| \\
                  & = & |(\phi_0,  g(z\phi_0+q_0) - g(z \phi_0 + q_1))| \\
                  & = & |(\phi_0, |z\phi_0+q_0|^2 (z \phi_0 + q_0) - |z \phi_0 + q_1|^2 (z \phi_0 + q_1) )| \\
                & \lesssim & \int \phi_0 (|z|^2 \phi_0^2 + |q_0|^2 + |q_1|^2) |q_0 - q_1| \lesssim |z|^2 \| q_0 - q_1 \|_{L^2}.
\end{eqnarray}
As $a_i = H_0^{-1}(-P_c g_{i-1} + e_{i-1}' q_{i-1})$ for $i = 1,2$ and $\| H_0^{-1}\|_{L^2 \to H^2} \leq \infty$, we have
\begin{eqnarray}
  \| a_1 -a_2 \|_{H^2}
  & \lesssim & \|P_c (g_1 - g_0) + e_0' q_0 - e_1' q_1 \|_{L^2} \\
  & \lesssim & \| g_1 - g_0 \|_{L^2} + |e_0' - e_1'| \| q_0 \|_{L^2} + | e_1' | \| q_0 - q_1 \|_{L^2}.
\end{eqnarray}
Since 
\begin{eqnarray}
  \| g_1 - g_0 \|_{L^2} 
  & = & \| g(z \phi_0 + q_1) - g(z \phi_0 + q_0) \|_{L^2} \\
  & \lesssim & |z|^2 \| \phi_0^2 (q_1 - q_2)\|_{L^2} + |z| \| \phi_0 (q_1^2 - q_2^2)\|_{L^2} + \| q_1^3 - q_2^3 \|_{L^2}  \\
  & \lesssim & |z|^2 \| \phi_0^2 \|_{L^3} \| q_1 - q_2 \|_{L^6}
                                   + |z| \| \phi_0 \|_{L^6} \| q_1 + q_2 \|_{L^6} \| q_1 - q_2 \|_{L^6} \\
                      &          & + \| (|q_1|^2 + |q_1 q_2| + |q_2|^2)\|_{L^4} \| q_1 - q_2 \|_{L^4},
\end{eqnarray}
together, we have
\begin{eqnarray}
  \| a_1 -a_2 \|_{H^2} \lesssim |z| \| q_1 - q_2 \|_{H^2} + |z|^2 |e'_0 - e'_1|.
\end{eqnarray}

Hence, $ M $ is a contraction mapping for $z$ sufficiently small.
Now by the contraction mapping theorem, there exists a unique fixed point $ (q, e') $ satisfying $ \| q \|_{H^2} = O(z^3) $ and $ |e'| = O(z^2) $ 
as $ z \rightarrow 0 $. The statements about derivatives of $ Q $ and $ E $ with respect to $ z $ follow by differentiating (\ref{l1e2}) with respect to $ z $ and applying the contraction mapping principle again.
%
%\end{proof}

\textit{Proof of exponential decay:}

\begin{lemma} \label{lExDecay}
For $ \ee > 0$, define the exponential weight function $ \chi_R$ by
\begin{equation} 
  \chi_{R,\ee} = 
    \begin{cases} 
      e^{\ee(|x|-R)} -1  & \mbox{if } R < |x| \leq 2R, \\
      e^{\ee(3R-|x|)} -1 & \mbox{if } 2R < |x| < 3R,  \\
      0    & \mbox{else}
    \end{cases}
.\end{equation}

Suppose for $ \ee > 0 $ small enough, $ f \in H^1 $ satisfies
\begin{equation} \| \chi_{R,\ee} f \|_{H^1} \leq C \end{equation}
for some constant $ C $ independent of $ R $, then
\begin{equation} e^{\ee' |x|}f \in H^1 \end{equation}
for some $ \ee' > 0 $.
\end{lemma}

\begin{proof}
For $ R > 0$, $ \| \chi_{R,\ee} f\|_{H^1} \leq C  $ implies that
\begin{equation} \| (e^{\ee(|x|-R)} -1 ) f \|_{H^1[\frac{3}{2} R,2R]} \leq C .\end{equation}
Since $ f \in H^1 $, 
\begin{equation} \| e^{\ee(|x|-R)} f \|_{H^1[\frac{3}{2} R,2R]} \leq C + \| f \|_{H^1} \leq C' .\end{equation}
$ e^{\frac{1}{2} \ee R} \leq e^{\ee(|x|-R)}$ for $ |x| \in [\frac{3}{2} R, 2R]$, so
\begin{equation} \| e^{\frac{1}{2} \ee R} f \|_{H^1[\frac{3}{2} R,2R]} \leq C' .\end{equation}
So 
\begin{equation} \| e^{(\frac{1}{2}(\frac{1}{2} \ee))(2R)} f \|_{H^1[\frac{3}{2} R,2R]} \leq C' .\end{equation}
Let $ \ee' = (\frac{1}{2}(\frac{1}{2} \ee )) $. Using $ e^{\ee' 2R} \geq e^{\ee' |x|} $ for $ |x| \in [\frac{3}{2} R,2R]$, we get that
\begin{equation} \label{*} \| e^{\ee' |x|} f \|_{H^1[\frac{3}{2} R,2R]} \leq C' \end{equation}
for some constant $ C' $ independent of $ R $.

Let $ \ee^{''} = \frac{1}{2} \ee' $. Then
\begin{equation} \| e^{\ee^{''} |x|}f \|_{H^1(|x| > 1)}^2 = \sum_{k=0}^{\infty} \| e^{\ee^{''} |x|} f \|^2_{H^1[\frac{2^{2k}}{3^{k}},\frac{2^{2(k+1)}}{3^{k+1}}]} .\end{equation}
Now, for each $k$, since $ e^{\ee'} = e^{\ee^{''}} e^{\ee^{''}} $, taking $ R = \frac{2^{2k+1}}{3^{k+1}}$ in (\ref{*}), we have
\begin{eqnarray} 
  C' 
  & \geq & \| e^{\ee' |x|} f \|_{H^1[\frac{2^{2k}}{3^{k}},\frac{2^{2(k+1)}}{3^{k+1}}]} \\
  &  =   & \| e^{\ee^{''} |x|} e^{\ee^{''} |x|} f \|_{H^1[\frac{2^{2k}}{3^{k}},\frac{2^{2(k+1)}}{3^{k+1}}]} \\
  & \geq & e^{(\ee^{''} \frac{2^{2k}}{3^{k}})} \| e^{\ee^{''} |x|} f \|_{H^1[\frac{2^{2k}}{3^{k}},\frac{2^{2(k+1)}}{3^{k+1}}]}. 
\end{eqnarray}
This means that,
\begin{equation}   
  \| e^{\ee^{''} |x|} f \|_{H^1[\frac{2^{2k}}{3^{k}},\frac{2^{2(k+1)}}{3^{k+1}}]} 
  \leq C' e^{-(\ee^{''} \frac{2^{2k}}{3^{k}})}
\end{equation}
Therefore,
\begin{eqnarray}
  \| e^{\ee^{''} |x|}f \|^2_{H^1(|x| > 1)} 
  &  =  & \sum_{k=0}^{\infty} \| e^{\ee^{''} |x|} f \|^2_{H^1[\frac{2^{2k}}{3^{k}},\frac{2^{2(k+1)}}{3^{k+1}}]} \\
  &  \leq  & C'^2 \sum_{k=0}^{\infty} e^{- \ee^{''} \frac{2^{2k+1}}{3^k} } \\
  &  <  & \infty
\end{eqnarray}

\end{proof}

%The following is the proof for Lemma \ref{l2}.

%\begin{proof} (Exponential decay in the sense that $ \| e^{\beta |x|} Q \|_{H^1} < \infty $)

By Lemma \ref{lExDecay}, to show that $ \| e^{\alpha |x|} Q \|_{H^1} < \infty $ for some $ \alpha > 0$, it suffices to show that $ \| \chi_{R,\ee} Q \|_{H^1} \leq C  $ for some constant $C$ independent of $R$. Here, $ \chi_{R,\ee} $ is the exponential weight function as in Lemma \ref{lExDecay}. 

Consider the bilinear form
\begin{equation} 
   \EE(\psi,\phi) = (\nabla \psi, \nabla \phi) 
                  + i \int (2 \bar{\psi} A \cdot \nabla \phi + \bar{\psi} (\nabla \cdot A) \phi) dx
                  + \int V \bar{\psi} \phi dx 
   \;\; \text{for} \; \psi, \phi \in H^1 
\end{equation}
associated to the magnetic Schr\"odinger operator $ - \Delta + 2 i A \cdot \nabla + i(\nabla \cdot A) + V $. Then
\begin{eqnarray}
  \EE(\psi, \psi) 
   & = & (\nabla \psi, \nabla \psi) 
          + i \int (2 \bar{\psi} A \cdot \nabla \psi + \bar{\psi} (\nabla \cdot A) \psi) dx
          + \int V \bar{\psi} \psi dx \\
   & = &  (\nabla \psi, \nabla \psi) 
          + 2 \Im( \int \bar{\psi} A \cdot \nabla \psi dx)
          + \int V \bar{\psi} \psi dx
\end{eqnarray}
Set
\begin{equation} 
  b:= \lim_{R \rightarrow \infty} \inf \{ \EE(\phi,\phi) | \phi \in H^1, \| \phi \|_2 = 1, \phi(x) = 0 \; \text{for} \; |x|<R \} 
.\end{equation}
We will show that $ b \geq 0 $ by contradiction. Suppose $ b < 0 $. Then there exists a sequence $ \phi_{R_j} \in H^1 $ with $ R_j \rightarrow \infty $, satisfying $ \| \phi_{R_j} \|_2 = 1 $, $ \phi_{R_j}(x) = 0 $ for $ 
|x| < R_j $, and $ \EE(\phi_{R_j}, \phi_{R_j}) < \delta $ for some fixed $ \delta < 0 $. 

Suppose $ V \in L^\infty$, then
\begin{equation} \int V \bar{\phi_{R_j}} \phi_{R_j} dx \leq  \| V \|_\infty \| \phi_{R_j} \|_2^2 = \| V \|_\infty .\end{equation}
Suppose $ V \in L^2 $, then
\begin{eqnarray}  
\int V \bar{\phi_{R_j}} \phi_{R_j} dx 
& \leq &  \| V \|_2 \| \phi_{R_j} \|_4^2 \\
& \lesssim &  \| V \|_2 \| \phi_{R_j} \|_2^{\frac{1}{2}} \| \nabla \phi_{R_j} \|_2^{\frac{3}{2}} \\
& \lesssim & \tilde{\dd} (\| \nabla \phi_{R_j} \|_2^{\frac{3}{2}})^{\frac{4}{3}} 
             + \frac{1}{\tilde{\dd}} (\| V \|_2 \| \phi_{R_j} \|_2^{\frac{1}{2}})^4 \\
& = & \tilde{\dd} \| \nabla \phi_{R_j} \|_2^2 + \frac{1}{\tilde{\dd}} \| V \|_2^4 .
\end{eqnarray}
Hence, 
\begin{equation} 
\int V \bar{\phi_{R_j}} \phi_{R_j} dx \lesssim \tilde{\dd} \| \nabla \phi_{R_j} \|_2^2 + \frac{1}{\tilde{\dd}} \| V \|_{L^\infty + L^2}
\;\; \text{where $\tilde{\dd}$ is sufficiently small}.
\end{equation}
Similarly, suppose $ A \in L^\infty $, then
\begin{equation} 
  |(\phi_{R_j}, A \cdot \nabla \phi_{R_j})| 
  \leq \| A \|_\infty \| \phi_{R_j} \|_2 \| \nabla \phi_{R_j} \|_2 
  = \| A \|_\infty \| \nabla \phi_{R_j} \|_2 .
\end{equation}
On the other hand, suppose $ A \in L^{(3+\tilde{\ee})} $, then 
\begin{eqnarray} 
|(\phi_{R_j}, A \cdot \nabla \phi_{R_j})| 
& \leq & \| A \|_{(3+\tilde{\ee})} \| \phi_{R_j} \|_{\frac{2(3+\tilde{\ee})}{1+\tilde{\ee}}} \| \nabla \phi_{R_j} \|_2 \\ 
& \lesssim & \| A \|_{(3+\tilde{\ee})} \| \phi_{R_j} \|_2^{\frac{5}{2} + \frac{3(1+\tilde{\ee})}{2(3+\tilde{\ee})}} 
             \| \nabla \phi_{R_j} \|_2^{\frac{5}{2} - \frac{3(1+\tilde{\ee})}{2(3+\tilde{\ee})}} \\
& = & \| A \|_{(3+\tilde{\ee})} \| \nabla \phi_{R_j} \|_2^{\frac{5}{2} - \frac{3(1+\tilde{\ee})}{2(3+\tilde{\ee})}} .
\end{eqnarray}
Hence, 
\begin{equation} 
|(\phi_{R_j}, A \cdot \nabla \phi_{R_j})| 
\lesssim  \| A \|_{L^{(3+\tilde{\ee})} + L^\infty} 
         ( \| \nabla \phi_{R_j} \|_2 + \| \nabla \phi_{R_j} \|_2^{\frac{5}{2} - \frac{3(1+\tilde{\ee})}{2(3+\tilde{\ee})}}) ,
\end{equation}
in which $ \frac{5}{2} - \frac{3(1+\tilde{\ee})}{2(3+\tilde{\ee})} $ is strictly less than 2 for $ \tilde{\ee} > 0 $.

Since $ \text{supp} (\phi_{R_j}) \subset \{ |x| \geq R_j \} $, by the assumption $ \| V_{-} \|_{(L^2+L^\infty)(|x| > R_j)} \rightarrow 0$ and $ \| A \|_{(L^{3+}+L^\infty)(|x| 
>R_j)} \rightarrow 0 $, $ \int V_- |\phi_{R_j}|^2 dx $ and the negative part of $\Im \int \overline{\phi_{R_j}} A \cdot \nabla 
\phi_R $ converge to 0. Hence, the negative part of the energy converges to 0, a contradiction. Thus $ b \geq 0 $. So 
there exists $ \delta(R) $ with $ \delta(R) \rightarrow b \geq 0 $ as $ R \rightarrow \infty $, such that for any $ 
\phi \in H^1 $ satisfying $ \phi(x) = 0 $ for $ |x| < R $, we have
\begin{equation} \EE(\phi, \phi) \geq \delta(R) \| \phi \|_2^2 .\end{equation}

For $ \phi \in H^1 $, we have
\begin{eqnarray}
  \delta(R) \| \chi_R \phi \|_2^2 
   & \leq & \EE(\chi_R \phi, \chi_R \phi) \\
   &  =   & (\nabla \chi_R \phi, \nabla \chi_R \phi) 
          - 2 \Im( \int \overline{\chi_R \phi} A \cdot \nabla \chi_R \phi dx)
          + \int V \overline{\chi_R \phi} \chi_R \phi dx .
\end{eqnarray}          
If we expand the factor $\nabla \chi_R \phi$, we get that 
\begin{eqnarray}
  (\nabla \chi_R \phi, \nabla \chi_R \phi)
  & = & (\phi \nabla \chi_R, \phi \nabla \chi_R) 
             + 2 (\phi \nabla \chi_R, \chi_R \nabla \phi) 
             + (\chi_R \nabla \phi, \chi_R \nabla \phi)
\end{eqnarray}
and since $\Im(\int |\phi|^2 A \cdot \chi_R^2 \nabla \chi_R ) = 0$
\begin{eqnarray}
  - 2 \Im( \int \overline{\chi_R \phi} A \cdot \nabla \chi_R \phi dx)
  & = & - 2 \Im(\int \chi_R^2 \overline{\phi} A \cdot \nabla \phi ) 
        - 2 \Im(\int |\phi|^2 A \cdot \chi_R^2 \nabla \chi_R ) \\
  & = & - 2 \Im(\int \chi_R^2 \overline{\phi} A \cdot \nabla \phi ).     
\end{eqnarray}
Since 
\begin{eqnarray}
2 (\phi \nabla \chi_R, \chi_R \nabla \phi) + (\chi_R \nabla \phi, \chi_R \nabla \phi)
- 2 \Im(\int \chi_R^2 \overline{\phi} A \cdot \nabla \phi )
+ \int V \overline{\chi_R \phi} \chi_R \phi dx
\end{eqnarray}
is nothing but $\EE(\chi_R^2 \phi, \phi)$, we have 
\begin{eqnarray}
  \delta(R) \| \chi_R \phi \|_2^2
  & \leq & \EE(\chi_R^2 \phi, \phi) +  \| \phi \nabla \chi_R \|_2^2 \\
  &  =   & (\chi_R^2 \phi, H_0 \phi) + e_0 \| \chi_R \phi \|_2^2 
            +  \| \phi \nabla \chi_R \|_2^2
\end{eqnarray}
where $ H_0 = - \Delta + i(A \cdot \nabla + \nabla \cdot A) + V - e_0 $.

From direct calculation, we see that for $ R > 0 $,
\begin{equation} | \nabla \chi_R | \lesssim \ee (\chi_R + 1) ,\end{equation}
so 
\begin{equation} \| \phi \nabla \chi_R \|_2^2 \lesssim \ee^2 \| \phi (\chi_R + 1) \|_2^2 .\end{equation}
Putting everything together, we have
\begin{eqnarray}
  \delta(R) \| \chi_R \phi \|_2^2
  \lesssim (\chi_R^2 \phi, H_0 \phi) + (e_0 + \ee^2) \| \chi_R \phi \|_2^2 
            +  \ee^2 \| \phi \|_2^2.
\end{eqnarray}

Since $ e_0 < 0 $ and $ \lim_{R \to \infty} \delta(R) \geq 0 $, for $ \ee $ small enough and $ R $ sufficiently large, $ \delta(R) - e_0 - \ee^2 $ is positive and bounded away from zero. Therefore, we have
\begin{equation} \| \chi_R \phi \|_2^2 \lesssim (\chi_R^2 \phi, H_0 \phi) + \ee^2 \| \phi \|_2^2 .\end{equation}

Next, 
\begin{eqnarray}
  \| \chi_R \nabla \phi \|_2^2
  &   \leq   & \| \nabla(\chi_R \phi) \|_2^2 + \| \phi \nabla \chi_R \|_2^2 \\
  & \lesssim & \EE(\chi_R \phi, \chi_R \phi) + 2 \Im( \int \overline{\chi_R \phi} A \cdot \nabla \chi_R \phi dx)
                - \int V \overline{\chi_R \phi} \chi_R \phi dx \\
  &          &  + \ee^2 \| \phi \|_2^2 \\       
\end{eqnarray}
Since 
\begin{equation}
  \Im( \int \overline{\chi_R \phi} A \cdot \nabla \chi_R \phi dx)
  \leq \| A \|_{L^\infty(|x| \geq R)} \| \chi_R \phi \|_{L^2} \| \nabla (\chi_R \phi) \|_{L^2}
\end{equation}
and 
\begin{eqnarray}
  \Im( \int \overline{\chi_R \phi} A \cdot \nabla \chi_R \phi dx)
  &  \leq  & \| A \|_{L^3 (|x| \geq R)} \| \chi_R \phi \|_{L^6} \| \nabla (\chi_R \phi) \|_{L^2} \\
  &  \leq  & \| A \|_{L^3 (|x| \geq R)} \| \chi_R \phi \|_{H^1} \| \nabla (\chi_R \phi) \|_{L^2} ,
\end{eqnarray}
we have that
\begin{eqnarray}
  \Im( \int \overline{\chi_R \phi} A \cdot \nabla \chi_R \phi dx)
  &  \leq  & \| A \|_{(L^\infty + L^3)(|x| \geq R)} \| \chi_R \phi \|_{H^1} \| \nabla (\chi_R \phi) \|_{L^2} \\
  &  \leq  & \| A \|_{(L^\infty + L^3)(|x| \geq R)} \| \chi_R \phi \|_{H^1}^2 .
\end{eqnarray}

Therefore,
\begin{equation} 
  \| \chi_R \nabla \phi \|_2^2 \lesssim \EE(\chi_R \phi, \chi_R \phi) + \| A \|_{(L^\infty + L^3)(|x| \geq R)} \| \chi_R \phi \|_{H^1}^2 
                                       + \| \chi_R \phi \|_2^2 + \ee^2 \| \phi \|_2^2   . 
\end{equation}

Now using $ \EE(\chi_R \phi, \chi_R \phi) = (\chi_R^2 \phi, H_0 \phi ) + e_0 \| \chi_R \phi \|_2^2 $ 
and $ \| \chi_R \phi \|_2^2 \lesssim (\chi_R^2 \phi, H_0 \phi) + \ee^2 \| \phi \|_2^2$, we have that
\begin{equation} 
  \| \chi_R \nabla \phi \|_2^2 
  \lesssim (\chi_R^2 \phi, H_0 \phi ) 
  + \| A \|_{(L^\infty + L^3)(|x| \geq R)} \| \chi_R \phi \|_{H^1}^2 
  + \ee^2 \| \phi \|_2^2 
.\end{equation}

Since
\begin{equation}
  \| \nabla (\chi \phi) \|_{L^2} = \| \phi \nabla \chi_R \|_{L^2} + \| \chi_R \nabla \phi \|_{L^2}
                                 \lesssim \ee \| \phi (\chi_R + 1) \|_{L^2} + \| \chi_R \nabla \phi \|_{L^2},
\end{equation}
putting everything together, we have that
\begin{equation} \| \chi_R \phi \|_{H^1}^2 \lesssim (\chi_R \phi, \chi_R H_0 \phi) + \ee^2 \| \phi \|_2^2 
                                      + \| A \|_{(L^\infty + L^3)(|x| \geq R)} \| \chi_R \phi \|_{H^1}^2 ,\end{equation}
so for $ R $ sufficiently large,
\begin{equation} \| \chi_R \phi \|_{H^1}^2 \lesssim (\chi_R \phi, \chi_R H_0 \phi) + \ee^2 \| \phi \|_2^2  .\end{equation}
If we let $ \phi = \phi_0 $ and use that $ H_0 \phi_0 = 0$, we have 
\begin{equation} \label{**} \| \chi_R \phi_0 \|_{H^1}^2 \lesssim \| \phi_0 \|_2^2 = 1 . \end{equation}
Next, let $ \phi = q $. Using that $ H_0 q = -P_c g(z \phi_0 + q) + e'q $, we get
\begin{eqnarray}
  \| \chi_R q \|_{H^1}^2 
  & \lesssim & (\chi_R q, \chi_R H_0 q) + \ee^2 \| q \|_2^2 \\
  & \lesssim & (\chi_R q, \chi_R (-P_c g(z \phi_0 + q) + e'q)) + \ee^2 \| q \|_2^2 \\
  & \lesssim & \| \chi_R^2 \; q \; g(z \phi_0 + q) \|_1 + e' \| \chi_R q \|_2^2 + \ee^2 \| q \|_2^2.
\end{eqnarray}
As $g(z) = |z|^2 z$, we have
\begin{eqnarray}
  &          & \| \chi_R^2 \; q \; g(z \phi_0 + q) \|_1 \\
  & \lesssim &  |z|^3 \| \chi_R^2 q \phi_0^3 \|_1 
               + |z|^2 \| \chi_R^2 q^2 \phi_0^2 \|_1 
               + |z| \| \chi_R^2 q^3 \phi_0 \|_1 
               + \| \chi_R^2 q^4 \|_1 \\
  & \lesssim & |z|^3 \| \chi_R^2 \phi_0^3 \|_2 \| q \|_2 
               + |z|^2 \| \chi_R^2 \phi_0^2 \|_2 \| q^2 \|_2 
               + |z| \| \chi_R^2 \phi_0 \|_2 \| \chi_R q^3 \|_2 \\
  &          & + \| \chi_R^2 q^2 \|_2 \| q^2 \|_2 \\
  & \leq     & o(z^2).       
\end{eqnarray}
Hence,
\begin{equation} \| \chi_R q \|_{H^1}^2 \leq o(z^2) \end{equation}
by (\ref{**}) and $\| q \|_{H^2} = o(z^2)$.

Next if we substitute $ \phi = D q$, and use that
\begin{equation} H_0 Dq = -P_c Dg(z \phi_0 + q) + q D e' + e' Dq , \end{equation}
we get
\begin{eqnarray}
  &          & \| \chi_R Dq \|_{H^1}^2 \\
  & \lesssim & (\chi_R Dq, \chi_R H_0 Dq) + \ee^2 \| Dq \|_2^2 \\
  & \lesssim & (\chi_R Dq, \chi_R (-P_c Dg(z \phi_0 + q) + q D e' + e' Dq)) + \ee^2 \| Dq \|_2^2 \\
  & \lesssim & \| \chi_R^2 \; Dq \; Dg(z \phi_0 + q) \|_1 + \| \chi_R^2 \; Dq \; q \; D e' \|_1 + e' \| \chi_R Dq \|_2^2 + \ee^2 \| q \|_2^2 .
\end{eqnarray}
Here, the first term $\| \chi_R^2 \; Dq \; Dg(z \phi_0 + q) \|_1$ is bounded by
\begin{eqnarray}
  &          & \| \chi_R^2 \; Dq \; Dg(z \phi_0 + q) \|_1 \\
  & \lesssim & \| \chi_R^2 \; Dq \; \phi_0 |z \phi_0 + q|^2 \|_1 \\
  & \lesssim & z^2 \| \chi_R^2 Dq \phi_o^3 \|_1 +  z \| \chi_R^2 Dq \phi_o^2 q \|_1 + \| \chi_R^2 \phi_o  q^2 \|_1 \\
  & \lesssim & z^2 \| \chi_R^2 Dq \phi_o^3 \|_1 +  z \| \chi_R^2 Dq \phi_o^2 q \|_1 + \| \chi_R^2 \phi_o  q^2 \|_1  \\
  & \lesssim & z^2 \| \chi_R Dq \|_{H^1} \| \chi_R \phi_0 \|_{H^1} \| \phi_0 \|_{H^1}^2 +  z \| Dq \|_{H^1} \| q \|_{H^1} \| \chi_R \phi_0 \|_{H^1}^2 \\
  &   \leq   & o(z^2),
\end{eqnarray}
and the second term $\| \chi_R^2 \; Dq \; q \; D e' \|_1$ is bounded by
\begin{eqnarray}
  \| \chi_R^2 \; Dq \; q \; D e' \|_1
  &   \leq   & \| \chi_R Dq \|_3 \| \chi_R q \|_3 \| De' \|_3 \\
  & \lesssim & \| \chi_R Dq \|_{H^1} \| \chi_R q \|_{H^1} \| De' \|_{H^1} \\
  &   \leq   & o(z^2).
\end{eqnarray}
Therefore, 
\begin{equation} \| \chi_R Dq \|_{H^1}^2  \leq o(z^2) .\end{equation}

Hence, by Lemma \ref{lExDecay} and $ Q = z \phi_0 + q $,  we have $ \| e^{\beta |x|} Q \|_{H^1} \leq \infty $ and $ \| e^{\beta |x|} DQ \|_{H^1} \leq \infty $ for some $ \beta > 0 $.

Next, we would like to show $ \| e^{\beta |x|} Q \|_{L^\infty} \leq \infty $ by bounding $ \| \Delta (e^{\beta |x|} Q) \|_{L^{\frac{3}{2}+}} $. Since $ \| \Delta (e^{\beta |x|} Q) \|_{L^{\infty}(|x|\leq 1)} < \infty $ already holds, it remains to show $ \| \Delta (e^{\beta |x|} Q) \|_{L^{\frac{3}{2}+}(|x| > 1)} < \infty $.
 Let $ \gamma = \frac{\beta}{3} $. Using the equation for $ Q $, we get
\begin{eqnarray}
 &          & \| \Delta (e^{\gamma |x|} Q) \|_{L^{\frac{3}{2}+}(|x| > 1)} \\
 & \lesssim & \| (\Delta e^{\gamma |x|}) Q \|_{L^{\frac{3}{2}+}(|x| > 1)} 
                + \| (\nabla e^{\gamma |x|}) \cdot (\nabla Q ) \|_{L^{\frac{3}{2}+}(|x| > 1)} \\
 &          &   + \| e^{\gamma |x|} A \cdot \nabla Q \|_{L^{\frac{3}{2}+}(|x| > 1)} 
                + \| e^{\gamma |x|} [(\nabla \cdot A) + V ]Q \|_{L^{\frac{3}{2}+}(|x| > 1)} \\
 &          &   + \| e^{\gamma |x|} g(Q) \|_{L^{\frac{3}{2}+}(|x| > 1)} + \| e^{\gamma |x|} EQ  \|_{L^{\frac{3}{2}+}(|x| > 1)} .
\end{eqnarray}
Let $ f $ and $g$ be such that $ \Delta e^{\gamma |x|} = f(x) e^{\gamma |x|}$ and $\nabla e^{\gamma |x|} = g(x) e^{\gamma |x|}$. We can bound the first two terms loosely by
\begin{equation}
  \| (\Delta e^{\gamma |x|}) Q \|_{L^{\frac{3}{2}+}(|x| > 1)} \lesssim \| e^{- \frac{2}{3} \gamma |x|} f(x) \|_{L^{6+}(|x| > 1)} \| e^{\beta |x|} Q \|_{L^{2}} 
\end{equation}
and
\begin{eqnarray} 
  &          & \| (\nabla e^{\gamma |x|}) \cdot (\nabla Q ) \|_{L^{\frac{3}{2}+}(|x| > 1)} \\
  & \lesssim & \| e^{- \frac{1}{3} \beta |x|} g(x) \|_{L^{6+}(|x| > 1)} \| e^{\frac{2}{3} \beta |x|} (\nabla Q ) \|_{L^{2}} \\
  & \lesssim & \|  e^{\frac{2}{3} \beta |x|} Q \|_{H^1} + \|  e^{-\frac{1}{3} \beta |x|} g(x) \|_{L^\infty(|x| > 1)} \|  e^{\beta |x|} Q \|_{L^2} .
\end{eqnarray}
Using similar ways, we can also bound  $ \| e^{\gamma |x|} g(Q) \|_{L^{\frac{3}{2}+}(|x| > 1)} $ and $ \| e^{\gamma |x|} EQ  \|_{L^{\frac{3}{2}+}(|x| > 1)} $.

Next, for $\| e^{\gamma |x|} A \cdot \nabla Q \|_{L^{\frac{3}{2}+}(|x| > 1)}$, we have
\begin{eqnarray}
&      & \| e^{\gamma |x|} A \cdot \nabla Q \|_{L^{\frac{3}{2}+}(|x| > 1)} \\
& \leq & \| A \|_{L^{3+} + L^\infty} 
          (\| e^{\frac{1}{3} \beta |x|} \nabla Q \|_{L^3(|x| > 1)} 
         + \| e^{\frac{1}{3} \beta |x|} \nabla Q \|_{L^{\frac{3}{2}}(|x| > 1)}) \\
& \lesssim & \| e^{\frac{-2}{3} \beta|x| } e^{\beta |x|} (\nabla Q)  \|_{L^3}
              + \| e^{\beta |x|} \nabla Q \|_{L^2}  .    
\end{eqnarray}
We already shown above that $ \| e^{\beta |x|} \nabla Q \|_{L^2} < \infty $. To bound $ \| e^{\frac{-2}{3} \beta|x| } e^{\beta |x|} (\nabla Q) \|_{L^3} $, let $ h = e^{\beta |x|} (\nabla Q) $ and from above, we know that $ h \in L^2 $. Now, consider the set 
\begin{equation} 
 M = \{ x | (e^{-\frac{2}{3} \beta |x|} |h|)^3 > |h|^2 \} 
   = \{ x | |h| >  e^{ 2 \beta |x|} \}. 
\end{equation}
Clearly, 
\begin{equation}
  \| e^{\frac{-2}{3} \beta|x| } e^{\beta |x|} (\nabla Q)  \|_{L^3(M^c)} = 
  \| e^{\frac{-2}{3} \beta|x| } h \|_{L^3(M^c)} 
  \leq \| |h|^{\frac{2}{3}} \|_{L^3} = \| h \|_{L^2}^{\frac{2}{3}} < \infty .
\end{equation} 
On the other hand, inside $M$, $ |e^{\beta |x|} (\nabla Q)| > e^{2 \beta |x|} $ and hence, $ |\nabla Q| > e^{\beta |x|} $. Then
\begin{eqnarray}
  \| e^{\frac{-2}{3} \beta|x| } e^{\beta |x|} (\nabla Q)  \|_{L^3(M)} 
  & \leq & \| e^{\frac{-2}{3} \beta|x|} |\nabla Q|^2 \|_{L^3}  \\
  & \leq & \| |\nabla Q|^2 \|_{L^3} \\
  &  =   & \| \nabla Q \|_{L^6}^3 \\
  &  \lesssim & \| \nabla Q \|_{H^1}^3 .
\end{eqnarray}

Hence, we have
\begin{equation} \| \Delta (e^{\gamma |x|} Q) \|_{L^{\frac{3}{2}+}} \leq \infty .\end{equation}
By Sobolev embedding, we have
\begin{equation} \| e^{\gamma |x|} Q \|_{L^{\infty}} \leq \infty  .\end{equation}
%%%%%%%%%%%%%%%%%%%%%%
%
%\end{proof}
%
\section{Proof of $ \| \phi \|_{W^{2,p}} \sim \| H_1 \phi \|_{L^p} $}
Recall that $ H_1 = H + K  = - \Delta + i(2 A \cdot \nabla + \nabla \cdot A) + V + K $.
Let
\begin{equation} W = i \nabla \cdot A + V + K.\end{equation}
Then
\begin{equation}  
  \| H_1 \phi \|_{L^p} \leq \| \Delta \phi \|_{L^p} + \| W \phi \|_{L^p} + 2 \| A \cdot \nabla \phi \|_{L^p} \lesssim \| \phi \|_{W^{2,p}}.
\end{equation}
In the above, we bounded $ \| A \cdot \nabla \phi \|_{L^p} $ by $ \| A \|_{L^{\infty}} \| \nabla \phi \|_{L^p}$.
Next,
\begin{eqnarray}
  \| \phi \|^2_{W^{2,p}}
  & \lesssim & \| \Delta \phi \|_{L^p}^2 + \| \phi \|_{L^p}^2 \\
  &     =    & \| (-H_1 + W + i2 A \cdot \nabla) \phi \|^2_{L^p} + \| \phi \|^2_{L^p} \\
  & \lesssim & \| H_1 \phi \|^2_{L^p} + \| W \phi \|_{L^p}^2 + 2 \| A \cdot \nabla \phi \|^2_{L^p} + \| \phi \|^2_{L^p} \\ 
  & \lesssim & \| H_1 \phi \|^2_{L^p} + \| W \|^2_{\infty} \| \phi \|_{L^p}^2 
                 + 2 \| \nabla A \|^2_{\infty} \| \phi \|^2_{L^p} + \| \phi \|^2_{L^p} \\ 
  & \lesssim & \| H_1 \phi \|^2_{L^p} + \| \phi \|^2_{L^p} .
\end{eqnarray}
To bound $ \| \phi \|^2_{L^p} $,consider
\begin{equation}
  (| \phi |^{p-2} \phi, H_1 \phi) = (|\phi|^{p-2} \phi, - \Delta \phi) + \int W |\phi|^p + \int |\phi|^{p-2} \phi A \cdot \nabla \bar{\phi}.
\end{equation}
Taking real parts on both sides, we get
\begin{equation}  
  \Re (| \phi |^{p-2} \phi, H_1 \phi) 
  - 2 \Re \int |\phi|^{p-2} \phi A \cdot \nabla \bar{\phi} 
  - \Re (|\phi|^{p-2} \phi, - \Delta \phi)
  \geq \int W |\phi|^p \geq C \| \phi \|^p_{L^p}
\end{equation}
by choosing $ K \geq \| V \|_\infty + \| \nabla \cdot A \|_\infty + C + 1$ where $ C $ is a large constant that will be used below. Using that 
\begin{eqnarray} 
   \int |\phi|^{p-2} |\nabla \phi|^2 + (p-2) \int |\phi|^{p-4} |\Re (\bar(\phi) \nabla \phi)|^2
  & = & \Re \int \nabla(|\phi|^{p-2} \bar{\phi}) \cdot \nabla \phi \\
  & = & \Re (|\phi|^{p-2} \phi, - \Delta \phi) \geq 0 ,
\end{eqnarray}
we get
\begin{eqnarray}
  C \| \phi \|^p_{L^p}
  & \lesssim & \| H_1 \phi \|_{L^p} \| |\phi|^{p-1} \|_{L^{\frac{p}{p-1}}} 
               + \| A \|_{\infty} \| |\phi|^{p-2} \phi \nabla \bar{\phi} \|_{L^1} \\
  &          & - \int |\phi|^{p-2} |\nabla \phi|^2 - (p-2) \int |\phi|^{p-4} |\Re (\bar(\phi) \nabla \phi)|^2  \\
  & \lesssim & \| H_1 \phi \|_{L^p} \| \phi \|^{p-1}_{L^p} 
               + \| A \|_{\infty} \| |\phi|^{\frac{p-2}{2}} \nabla \phi \|_{L^2}  \| \phi^{\frac{p}{2}} \|_{L^2}
               - \| |\phi|^{\frac{p-2}{2}} |\nabla \phi| \|_{L^2} \\
  & \lesssim & \| H_1 \phi \|_{L^p} \| \phi \|^{p-1}_{L^p} 
               + \ee^2 \| |\phi|^{\frac{p-2}{2}} \nabla \phi \|_{L^2} 
               + \frac{1}{\ee^2} \| A \|_{\infty} \| \phi \|_{L^p}^p \\
  &          & - \| |\phi|^{\frac{p-2}{2}} |\nabla \phi| \|_{L^2} .
\end{eqnarray}
Now if we choose $ \ee $ small enough, we have
\begin{equation} C \| \phi \|^p_{L^p} 
  \lesssim \| H_1 \phi \|_{L^p} \| \phi \|^{p-1}_{L^p} 
  + \frac{1}{\ee^2} \| A \|_{\infty} \| \phi \|_{L^p}^p .
\end{equation}
Dividing by $ \| \phi \|^{p-1}_{L^p} $, we have 
\begin{equation} C \| \phi \|_{L^p} \lesssim \| H_1 \phi \|_{L^p} + \frac{1}{\ee^2} \| A \|_{\infty} \| \phi \|_{L^p} .\end{equation}
Finally, if we choose $ C $ large enough and put everything together, we have
\begin{equation} \| \phi \|_{W^{2,p}} \lesssim \| H_1 \phi \|_{L^p}  .\end{equation}
\section*{Acknowledgement}
The author would like to thank Stephen Gustafson and Tai-Peng Tsai for many helpful discussions.
\bibliographystyle{plain}

\end {document}